\documentclass{amsart}
\usepackage{amssymb}
\usepackage{enumitem}
\usepackage{xcolor}
\usepackage{pinlabel}
\usepackage{hyperref}
\usepackage{thm-restate}
\usepackage{fixltx2e}

\newtheorem {theorem}{Theorem}[section]
\newtheorem {lemma} [theorem] {Lemma}
\newtheorem {proposition} [theorem] {Proposition}
\newtheorem {corollary} [theorem] {Corollary}

\newtheorem {claim}{Claim}
\newtheorem* {claim*}{Claim}

\newtheorem* {subclaim*}{Subclaim}

\newtheorem {convention} [theorem] {Convention}

\theoremstyle{definition}
\newtheorem{remark}[theorem]{Remark}
\newtheorem {definition} [theorem] {Definition}
\newtheorem {notation} [theorem] {Notation}
\newtheorem {example} [theorem] {Example}

\newcommand{\mc}{\mathcal}
\newcommand{\from}{\colon\thinspace}
\newcommand{\st}{:}
\newcommand{\Homeo}{\operatorname{Homeo}}

\newcommand{\Hom}{\operatorname{Hom}}
\newcommand{\acts}{\curvearrowright}
\newcommand{\dvis}{d_\partial}
\newcommand{\dgam}{d_{\Gamma}}
\newcommand{\dcusp}{d_{X}}
\newcommand{\Izero}{I \cup \{0\}}

\newcommand{\bgamp}{\partial(\Gamma, \mathcal{P})}
\newcommand{\diam}{\operatorname{diam}}
\newcommand{\Cay}{\operatorname{Cay}}

\newcommand{\idgamma}{\mathit{id}}
\newcommand{\PSL}{\mathrm{PSL}}

\newcommand{\minus}{-}
\newcommand{\bN}{\mathbb{N}}
\newcommand{\bH}{\mathbb{H}}
\newcommand{\bR}{\mathbb{R}}

\newcommand{\depth}{D_X}
\newcommand{\coding}[1]{{\mathbf{#1}}}

\newcommand{\lab}{\operatorname{\mathbf{Lab}}}
\newcommand{\head}{\iota}
\newcommand{\tail}{\tau}

\newcommand{\semiconj}{\mathcal{R}}
\newcommand{\bigepsilon}{\varepsilon}
\newcommand{\smallepsilon}{\varepsilon'}

\newcounter{notes}%[page]   %Le 2eme argument fait reinitialiser les numeros de notes a chaque page
\newcommand{\putinbox}[1]{\noindent\fbox{\parbox{.98\textwidth}{#1}}}

\title[Topological stability of relatively hyperbolic groups]{Topological stability of relatively hyperbolic groups acting on their boundaries}
\author[K. Mann]{Kathryn Mann}
\address{Department of Mathematics, Cornell University, Ithaca, NY 14853, USA
}
\email{k.mann@cornell.edu}

\author[J.F. Manning]{Jason Fox Manning}
\address{Department of Mathematics, Cornell University, Ithaca, NY 14853, USA
}
\email{jfmanning@cornell.edu}

\author[T. Weisman]{Theodore Weisman}
\address{Department of Mathematics, University of Michigan, Ann Arbor,
  MI 48109, USA
}
\email{tjwei@umich.edu}

\begin{document}

\begin{abstract}
We prove a topological stability result for the actions of relatively hyperbolic groups on their Bowditch boundaries.  More precisely, we show that a sufficiently small perturbation of the standard boundary action, if assumed on each parabolic subgroup to be a perturbation by semi-conjugacy, is in fact always globally semi-conjugate to the standard action.  This proves a relative version of the main result of \cite{MMW1}.  The assumption of control on the perturbation of parabolics is necessary.
\end{abstract}

\maketitle

\section{Introduction} 
An action of a group $G$ on a topological space $X$ is said to be {\em topologically} or {\em $C^0$ stable} if small perturbations of the action ``retain the same complexity of dynamical information".  Formally, an action $\rho_0: G \to \Homeo(X)$ is {\em $C^0$ stable} if there exists a neighborhood $U$ of $\rho_0$ in the compact-open topology on $\Hom(G, \Homeo(X))$ so that every $\rho \in U$ is related to $\rho_0$ by a continuous, surjective map $h: X \to X$ satisfying $h \circ \rho = \rho_0 \circ h$.  This map $h$ is called a {\em semi-conjugacy}, and $\rho_0$ is said to be a {\em (topological) factor} of $\rho$.  
A frequent theme in dynamics is that, at least for sufficiently rich actions, some control on the differentiability or regularity of the action $\rho$ can force any such map $h$ to be invertible (so a genuine conjugacy), while in the purely $C^0$ case a semi-conjugacy is the best one can hope for.  

This paper treats the question of stability of certain boundary actions of groups.  The general study of stability of boundary actions  has a long history.  Sullivan \cite{Sullivan} proved that the action of a Kleinian group on its limit set is stable in the sense of $C^1$ dynamics, meaning that $C^1$-close actions remain conjugate.  This is one instance of the theme of control on regularity described above.  Sullivan's work was generalized by Kapovich--Kim--Lee \cite{KKL} to the broader class of what they call ``meandering hyperbolic actions" under Lipschitz--close perturbations, examples of which include the boundary actions of hyperbolic groups and of uniform lattices.   

Topological stability of hyperbolic groups acting on their boundaries was shown in \cite{BowdenMann, MannManning, MMW1}; see also \cite{GromovEssay}.  The approach from \cite{BowdenMann} was very recently adapted (with much additional work) to establish topological stability for actions of uniform lattices on Furstenburg boundaries in \cite{CINS}.   The current paper treats the case of {\em relatively hyperbolic} group pairs.  Such a pair has a canonical action on a compact metrizable space called the \emph{Bowditch boundary}, whose definition is recalled in Section~\ref{sec:setup} below.  
We prove the following. 

\begin{restatable}[$C^0$ stability for relatively hyperbolic groups]{theorem}{mainthm}\label{thm:main} 
  Let $\Gamma$ be hyperbolic relative to $\mathcal{P}$ and $\rho_0$
  the natural action of $\Gamma$ on the Bowditch boundary $\bgamp$.

  For any neighborhood $\mathcal{V}$ of the identity in $C(\bgamp)$,
  there exists a neighborhood $\mathcal{U}$ of $\rho_0$ in
  $\Hom(\Gamma, \Homeo(\bgamp))$ and a neighborhood $\mathcal{V'}$ of
  the identity in $C(\bgamp)$ such that, if $\rho \in \mathcal{U}$ is
  an action whose restriction to each $P \in \mathcal{P}$ 
  has $\rho_0|_P$ as a topological factor, related by a semi-conjugacy in $\mathcal{V'}$, then
  $\rho$ has $\rho_0$ as a factor via a semi-conjugacy in
  $\mathcal{V}$.
\end{restatable}
As is standard, for a topological space $X$ (in our case $X = \bgamp$), the notation $C(X)$ refers to the space of continuous maps $X \to X$, equipped with the compact-open topology; and $\Hom(\Gamma, G)$ is the space of homomorphisms from $\Gamma$ to the topological group $G$ (in our case $G=\Homeo(\bgamp))$, also equipped with the compact-open topology.  

\subsection{Examples} 
The following examples show that control on the restriction of $\rho$ to each subgroup
$P \in \mathcal{P}$ is a necessary hypothesis.
In fact it is necessary even when $\Gamma$ is a lattice in a Lie group $G$, and the deformed action $\rho$ is induced by deforming the inclusion representation in the character variety.
\begin{example} 
  Consider a finite-area hyperbolic surface with a single cusp,
  isometric to a quotient $\bH^2 / \Gamma$ for a discrete subgroup
  $\Gamma \subset \PSL(2, \bR)$. Then $\Gamma$ is a relatively
  hyperbolic group, relative to the fundamental group of the cusp, and
  the action of this group on its Bowditch boundary is the action of
  $\Gamma$ on $\bR \mathrm{P}^1 = \partial \bH^2$.
  \begin{enumerate}[label=(\roman*)]
  \item Let $\rho:\Gamma \to \PSL(2, \bR)$ be a small deformation of
    the inclusion representation
    $\rho_0:\Gamma \hookrightarrow \PSL(2, \bR)$ for which a generator
    of the cusp group acts by an elliptic transformation. As elliptic
    transformations have no fixed points in $\partial \bH^2$, this
    deformed action cannot be semi-conjugate to the original boundary
    action.
  \item\label{item:h2_loxodromic_example} Now let
    $\rho:\Gamma \to \PSL(2, \bR)$ be a deformation of $\rho_0$ where
    the cusp group instead acts by a loxodromic transformation, which
    fixes a pair of points in $\partial \bH^2$ close to the parabolic
    fixed point for the original cusp group action. The map
    $\partial \bH^2 \to \partial \bH^2$ which collapses the small arc
    $A$ joining this pair of fixed points is a semi-conjugacy for the
    $\rho$-action of the cusp group, which is close to the
    identity. Additionally collapsing all of the arcs in the
    $\rho(\Gamma)$-orbit of $A$ yields a semi-conjugacy for the
    $\rho$-action of the whole group $\Gamma$. This is an example of the situation
    described by Theorem~\ref{thm:main}.
  \end{enumerate}
\end{example}

It is also possible to generalize both of the examples above to
higher dimensions.
\begin{example} 
  Suppose that $M$ is a finite-volume non-compact hyperbolic
  $n$-manifold, isometric to $\bH^n/\Gamma$ for a discrete group
  $\Gamma \subset \mathrm{Isom}(\bH^n) \simeq \mathrm{PO}(n, 1)$; as
  in the previous example, $\Gamma$ is then relatively hyperbolic,
  relative to its collection of cusp subgroups, and the action on its
  Bowditch boundary is the induced action on
  $\partial \bH^n \simeq S^{n-1}$.
  \begin{enumerate}[label=(\roman*)]
  \item When $n = 3$, Thurston's hyperbolic Dehn filling theorem
    implies that there are arbitrarily small deformations of the
    inclusion $\rho_0:\Gamma \hookrightarrow \mathrm{PO}(n,1)$ for
    which the action on $\bH^3$ (and hence the induced action on
    $\partial \bH^3$) has infinite kernel. These deformations cannot
    be semi-conjugate to the original action, where the kernel is
    trivial.
  \item 
    For each $n\ge 2$, there are examples of (non-uniform) lattices
    $\Gamma \subset \mathrm{PO}(n,1)$ which have deformations in
    $\Hom(\Gamma, \mathrm{PGL}(n+1, \bR))$ which are discrete,
    faithful, and preserve a convex open subset
    $\Omega \subset \bR \mathrm{P}^n$, arbitrarily close (but not
    equivalent) to the projective model for $\bH^n$; in some of these
    examples, the cusp groups preserve small $k$-simplices embedded in
    $\partial \Omega$, for some $1 \le k < n$ (see \cite{BDL18},
    \cite{BallasMarquis20}, \cite{Bobb19}). The induced action of
    $\Gamma$ on $\partial \Omega \simeq S^{n-1}$ then cannot be
    conjugate to the original action of $\Gamma$ on $\partial
    \bH^n$. However, it follows from \cite[Section 5]{Weisman23} that
    (for sufficiently small deformations) the map collapsing all of
    the $k$-simplices is a semi-conjugacy to the standard action of
    $\Gamma$ on its Bowditch boundary.
  \end{enumerate}
\end{example}

We also note that, in the broad context of $C^0$ deformations covered
by Theorem~\ref{thm:main}, 
it is possible to have arbitrarily small deformations $\rho$ of $\rho_0$ such that $\rho$ and $\rho_0$ are not conjugate (rather, they are only semi-conjugate), but such that their restrictions
to each peripheral subgroup are genuinely {\em conjugate}.  
In \cite[Section
4]{BowdenMann} and \cite[Example 1.4]{MMW1}, examples of small non-conjugate deformations 
are constructed in the case where $\mc{P}= \emptyset$; these can be generalized to the relatively hyperbolic setting.  

\subsection*{Outline}
The broad strategy of this work follows that of \cite{MMW1}, but much additional technical work is needed in the presence of parabolic elements.  We adapt and use tools and ideas from \cite{Weisman} and the geometry of relatively hyperbolic groups.  A reader looking for a gentler introduction to rigidity of boundary actions may wish to first read the simpler proof in \cite{MMW1}.

In Section \ref{sec:setup} we review necessary material on relatively hyperbolic groups and set the stage for the proof.  In Section \ref{sec:automaton} we define an automaton which codes boundary points, adapted to the setting of the desired stability theorem for relatively hyperbolic groups.  Essential properties of this automaton are proved in Section \ref{sec:automaton_prop}

Section \ref{sec:uniform_nest} is the technical heart of the paper.
The key proposition (Proposition \ref{prop:unif_nest}) is a ``uniform
nesting" condition for sequences of nested sets furnished by the
automaton, which can be translated into a stable condition under
perturbation that allows us to define the desired semi-conjugacy
between the standard boundary action and a sufficiently small
perturbation.  Section \ref{sec:proof} uses all the previous work to
define this semi-conjugacy and conclude the proof.

\subsection*{Acknowledgments}
KM was partially supported by NSF grant DMS 1844516 and a Sloan
fellowship.  JM was partially supported by the Simons Foundation, grant number 942496.  TW was partially supported by NSF grant
DMS-2202770.

%---------------------
\section{Set-up} \label{sec:setup} We assume the reader has basic
familiarity with the theory of relatively hyperbolic groups; general
background can be found in \cite{Bowditch12}, or \cite[Section 2]{GM08} and references therein. 
In this section we set
notation and recall the essential properties that we will use.

Let $\Gamma$ be a relatively hyperbolic group, relative to a finite
collection $\mathcal{P}$  of infinite subgroups. 
Fix a finite, symmetric generating set $\mathcal{S}$ for $\Gamma$. We let
$\Cay(\Gamma) = \Cay(\Gamma, \mathcal{S})$ denote the Cayley graph of
$\Gamma$ with respect to the generating set $\mathcal{S}$, and we let
$\dgam$ denote the metric on $\Cay(\Gamma)$ induced by this
generating set.  We also assume the generating set $\mathcal{S}$ is
compatible with $\mathcal{P}$, i.e. that for every parabolic subgroup
$P \in \mathcal{P}$, the intersection $P \cap \mathcal{S}$ is a
generating set for $P$. If $\Gamma$ is hyperbolic (i.e.
$\mathcal{P} = \emptyset$) then Theorem~\ref{thm:main} follows from
\cite{MMW1}, so we assume for the duration of this work that
$\mathcal{P} \neq \emptyset$. We also assume that
$(\Gamma, \mathcal{P})$ is non-elementary (meaning that
$\mathcal{P} \neq \{\Gamma\}$ and $\Gamma$ is not finite or virtually
cyclic), since in these cases the Bowditch boundary $\bgamp$ contains
at most two points and the theorem is trivial.

The Bowditch boundary $\bgamp$ can be identified with the Gromov boundary of a hyperbolic space $X = X(\Gamma, \mc{P}, \mathcal{S})$ which was defined in \cite{GM08} and called a \emph{cusped space} for the pair $(\Gamma,\mc{P})$.  The space $X$ is a locally finite metric graph, with each edge having length $1$.  We use $\dcusp$ to refer to the metric on $X$.  The group $\Gamma$ acts properly by isometries on $X$, and 
the Cayley graph $\Cay(\Gamma,\mc{S})$ embeds properly and $\Gamma$-equivariantly (though not quasi-isometrically) as a subgraph of $X$.

We will need to use some results on the geometry of horoballs in $X$ and
geodesics near horoballs, especially in
Section~\ref{sec:uniform_nest}.  To this end, define the \emph{depth}
$\depth$ of a vertex of $X$ to be its distance from the Cayley graph
in $X$ and extend linearly across edges to get a continuous function
$\depth\from X \to [0,\infty)$.  Thus, the Cayley graph in $X$ is the
set $\depth^{-1}(0)$; a \emph{horoball} is the smallest full subgraph
containing the closure of a component of $\depth^{-1}(0,\infty)$.  For
an integer $k > 0$, we refer to a component of $\depth^{-1}[k,\infty)$
as a \emph{$k$--horoball}. If $\mc{H}$ is a horoball, we say that a
geodesic $\gamma\from[a,b] \to \mc{H}$ is \emph{regular} if there are
$a \le A < B \le b$ with $B - A \le 3$ and
\begin{equation*}
  \frac{d}{dt}\depth\circ\gamma(t) =
  \begin{cases}
    1 & t<A\\
    0 & A<t<B\\
    -1 & t>B
  \end{cases}.\end{equation*}
A geodesic in a horoball is \emph{vertical} if it is regular and
either $a=A=B$ or $A=B=b$.  We will say a geodesic in $X$ is
\emph{regular} if every intersection with a horoball is
regular. Whenever necessary, we will assume our geodesics are regular.  We also fix an integer $\delta \ge 1$ so that the cusped space $X$ is $\delta$--hyperbolic (in the sense that all geodesic triangles are $\delta$--thin).

The following result was proved in \cite{GM08}.
\begin{lemma}\cite[Lemmas 3.10, 3.26]{GM08}\label{lem:gmlemma}
  For any $k \ge \delta + 1$, the $k$--horoballs are convex.
  Moreover, any geodesic in a horoball is Hausdorff distance at most
  $4$ from a regular unit speed geodesic $\gamma\from [a,b]\to X$ with
  the same endpoints.
\end{lemma}

We also need one more lemma about the geometry of geodesics which pass
through horoballs.
\begin{lemma}\label{lem:quickapproach}
  Let $\mc{H}$ be a $k$-horoball of $X$ for some $k \ge 0$, let
  $\sigma\from [a,b]\to X$ be a regular unit speed geodesic, and let
  $[a',b']$ be a connected component of $\sigma^{-1}(\mc{H})$.  Assume
  that $b'-a' \ge (4\delta + 3)$.  Then for $t\in [a,a']$ we have
  \[ (a'-t) - \delta \le d_X(\sigma(t),\mc{H}) \le a'-t. \]
\end{lemma}
\begin{proof}
  Let $x$ be a closest point in $\mc{H}$ to $\sigma(t)$, and let
  $y = \sigma(a')$.  Any geodesic from $\sigma(t)$ to $x$ can be
  extended by a vertical path to a point $x'$ on the
  $k + 2\delta$--horoball $\mc{H}'$ nested inside $\mc{H}$.  Call this
  extended geodesic $\tau$.  Now consider a geodesic triangle two of
  whose sides are $\tau$ and $\sigma(t,a'+2\delta)$.  (Regularity of
  $\sigma$ implies that $\sigma(a'+2\delta)$ is on the boundary of
  $\mc{H}'$.)  The third side of the triangle has endpoints in
  $\mc{H}'$, which is convex by the first part of
  Lemma~\ref{lem:gmlemma}.  In particular the distance from $y$ to
  this third side is at least $2\delta$, so there is a point $y'$ on
  $\sigma(t,a+2\delta)$ which is within $\delta$ of $y$.  Either $y'$
  lies between $\sigma(t)$ and $x$, or on $\tau\cap \mc{H}$ at depth
  at most $\delta$.  In either case, we deduce that
  $a'-t \le d_X(\sigma(t),x)+\delta$, as desired.
\end{proof}

The Gromov boundary of the cusped space $X$ is the Bowditch boundary
$\bgamp$ of the pair $(\Gamma, \mathcal{P})$ (see \cite{Bowditch12},
\cite{GM08}). We fix a metric $\dvis$ on $\bgamp$.  Metric notions
such as diameter, $\epsilon$-neighborhoods are always with respect to
this metric.  We write $B_r(x)$ for the (open) ball about $x$ of
radius $r$, and $N_r(Y)$ for the open $r$-neighborhood of a subset
$Y$.

When we need to distinguish the natural action of $\Gamma$ on $\bgamp$
from another action of this group on the space, we will use the
notation $\rho_0(g)(x)$ for the action of $g \in \Gamma$ on
$x \in \bgamp$.  However, in the first part of this work we use {\em
  only} this action, and so for convenience shorten this to $gx$.

Each subgroup in $\mathcal{P}$ acts on $\bgamp$ with a unique fixed
point. We denote the finite set of fixed points of groups in $\mc{P}$
by $\Pi$, and for $p \in \Pi$, we let $\Gamma_p \in P$ denote the
subgroup fixing $p$.  Thus,
\[ \mathcal{P} = \{ \Gamma_p \st p \in \Pi\}. \] Any
$\Gamma$-translate in $\bgamp$ of a point in $\Pi$ is called a
\emph{parabolic point}.

We now recall some important properties of the action
of $\Gamma$ on $\bgamp$, which we use to set up the next construction.
First, the group $\Gamma$ acts on $\bgamp$ as a \emph{convergence
  group} (see \cite{Tukia94}), meaning that the induced action on the space of distinct
triples in $\bgamp$ is properly discontinuous.

Secondly, $\Gamma$ acts cocompactly on pairs of distinct points in the
Bowditch boundary $\bgamp$.  Thus, we can set the following separation
constant:
\begin{definition}
  \label{def:pair_separation_constant}
  We fix a constant $D > 0$ such that for any pair of distinct points
  $x, y \in \bgamp$, there is some $g \in \Gamma$ such that
  $\dvis(g x, g y) > D$.
\end{definition}

Finally, the action of $\Gamma$ on $\bgamp$ is \emph{geometrically
  finite} in a dynamical sense, meaning that every non-parabolic point
in $\bgamp$ is a \emph{conical limit point}, and every parabolic point
is \emph{bounded}.

We recall that a point $z \in \bgamp$ is a conical limit point if
there are distinct points $a,b \in \bgamp$ and a sequence
$(g_i)_{i \in \bN}$ so that $g_iz \to b$ and $g_ix \to a$ uniformly on
compacts in $\bgamp \minus \{z\}$. A parabolic point $p$ is
\emph{bounded} if its stabilizer $\Gamma_p \subset \Gamma$ acts with
compact quotient on the space $\bgamp \minus \{p\}$.  Thus, we define
the following.
\begin{definition}
  \label{defn:parabolic_fund_domain}
  For each parabolic point $p \in \Pi$, we fix a compact set
  $K_p \subset \bgamp \minus \{p\}$ such that
  $\Gamma_p \cdot K_p = \bgamp \minus \{p\}$. 
\end{definition}

For each $p$, since $K_p$ is compact, the distance $\dvis(x, p)$ is bounded below
by a positive constant. Further, since we assume $\Gamma$ is
non-elementary, $\bgamp$ is uncountable. Since $\Gamma_p$ is
countable, $K_p$ must have positive diameter. So, we can also make the
following definition.
\begin{definition}
  \label{defn:bounded_parabolic_constant}
  We fix a constant $D_\Pi > 0$ so that for each $p \in \Pi$, we have $D_\Pi < \diam(K_p)$ and
  $D_\Pi < \dvis(K_p, p)$.
\end{definition}

%---------------------
\section{Automaton} \label{sec:automaton}

In this section we give a version of the construction in Section~2 of
\cite{MMW1} (see also Sections 5 and 6 of \cite{Weisman}). The former
paper shows that every point in the Gromov boundary of a hyperbolic
group has an ``expanded neighborhood'' which is well-adapted to the
construction of an automaton which codes points.  In the setting of
this paper, conical points can be coded in essentially the manner of
\cite{MMW1}. However, parabolic fixed points need separate treatment,
for which we adapt the approach in \cite{Weisman}.

We start with the conical case, adapting Lemma~2.3 in \cite{MMW1} as
follows.
\begin{lemma}[Expanded neighborhoods]\label{lem:goodneighborhoods}
  For any positive $\epsilon < \frac{D}{5}$ and any
  conical limit point $z \in \bgamp$, there exists $\alpha_z\in \Gamma$
  and a pair of open neighborhoods $V(z)\subset W(z)$ of $z$ so that
  \begin{enumerate}
  \item\label{lemitm:smalldiam} $\diam(W(z)) \le \epsilon$;
  \item\label{lemitm:bigW} $\diam(\alpha_z^{-1}W(z))>4\epsilon$; and
  \item\label{lemitm:deepnestU}
    $\overline{N}_{2\epsilon}(\alpha_z^{-1}V(z)) \subset
    \alpha_z^{-1}W(z)$.
  \end{enumerate}
\end{lemma}
\begin{proof}
    We choose some $\epsilon<\frac{D}{5}$ where $D$ is the constant from Definition~\ref{def:pair_separation_constant}.
  
    Let $z\in \bgamp$ be is a conical limit point.  Since $z$ is
    conical, we can find distinct points $a,b$ and a sequence of group
    elements $(g_i)_{i\in\bN}$ so that $g_iz\to b$ and $g_i x\to a$
    uniformly for all $x \neq z$.  Up to post-composing all $g_i$ with
    a fixed element $g$ as in
    Definition~\ref{def:pair_separation_constant} if necessary, we may
    assume $\dvis(a,b) \ge D$.  Also, since $\bgamp$ is perfect, there
    is some point $a'\ne a$ with $\dvis(a,a') = \epsilon' <\epsilon$.
  
    Let $W(z)=B_{\epsilon/2}(z)$, so Property~\eqref{lemitm:smalldiam} is satisfied.  Let $K$ be the complement of $W(z)$ in $\bgamp$.
    The set $K$ is compact and does not contain $z$, so for $i$ sufficiently large, we have $g_i K\subset B_{\epsilon'}(a)$ and $g_i z\in B_\epsilon(b)$.

  Fixing some such $i$, set $\alpha_z= g_i^{-1}$, and let
  $V(z) = \alpha_z(B_\epsilon(b))$.  Note that $B_\epsilon(a)$ contains
  $\alpha_z^{-1} K = \bgamp - \alpha_z^{-1} W(z)$.   
  Since $B_\epsilon(a)$ is
  disjoint from $B_\epsilon(b)$, we have
  \[ V(z) = \alpha_z(B_\epsilon(b))\quad \subset\quad  \bgamp - \alpha_z (K) = W_z. \]

  The set $\alpha_z^{-1}W(z) = \bgamp \minus \alpha_z^{-1}K$ contains both $b$ and $a'$, so $\diam(\alpha_z^{-1}W(z)) \ge \dvis(b,a') \ge D-\epsilon > 4\epsilon$, establishing Property~\eqref{lemitm:bigW}.

 Finally, since $\dvis(b,\alpha_z^{-1}K)\ge D-\epsilon > 4\epsilon$, we have
 \[ \overline{N}_{2\epsilon}(\alpha_z^{-1}V(z))
   \subset \overline{B}_{3\epsilon}(b) \subset (\bgamp \minus \alpha_z^{-1} K) = \alpha_z^{-1}W(z), \]
  establishing Property~\eqref{lemitm:deepnestU}.
  \end{proof}

  To treat parabolic points, we follow an argument given in
  \cite{Weisman}; compare the lemma below to \cite[Lemma
  6.7]{Weisman}.
  \begin{lemma}[Parabolic points]
  \label{lem:goodneighborhoods_p}
  For each point $p \in \Pi$, any $\epsilon < D_\Pi/5$, and each
  $q = gp$, there exist open sets
  $\hat{V}(p) \subset \hat{W}(p) \subset \bgamp \minus \{p\}$,
    and a finite set $F_q \subset g\Gamma_p$
  such that:
  \begin{enumerate}
  \item\label{lemitm:parabolic_bigW}
    $\diam(\hat{W}(p)) > 4\epsilon$;
  \item\label{lemitm:parabolic_deepnest}
    $\overline{N}_{2\epsilon}(\hat{V}(p)) \subset \hat{W}(p)$;
  \item\label{lemitm:parabolic_closure}
    $\overline{N}_{\epsilon}(\hat{W}(p))$ does not contain
    $p$; 
  \item\label{lemitm:parabolic_coding_elt} The set $V(q): = \{q\} \cup \bigcup_{\alpha \in g{\Gamma_p} - F_q} \alpha \hat{V}(p)$ is  a neighborhood of $q$;
    \item\label{lemitm:parabolics_in} The set $W(q): = \{q\} \cup \bigcup_{\alpha \in g{\Gamma_p} - F_q} \alpha \hat{W}(p)$ 
 contains $V(q)$; 
    \item\label{lemitm:parabolic_small_diam}
    $\diam(W(q)) \le \epsilon$.
  \end{enumerate}
\end{lemma}
\begin{proof}
  Let $q \in \bgamp$ be a parabolic point with $q = gp$ for
  $p \in \Pi$. Let $K_p$ be the compact set from
  Definition~\ref{defn:parabolic_fund_domain}, so that
  $\Gamma_p \cdot K_p = \bgamp \minus \{p\}$, $\diam(K_p) > D_\Pi$, and
  $\dvis(x, p) > D_\Pi$ for every $x \in K_p$. We let
  $\hat{V}(p) = N_\epsilon(K_p)$, and we let
  $\hat{W}(p) = N_{4\epsilon}(K_p)$.

  Conditions
  \eqref{lemitm:parabolic_bigW} and \eqref{lemitm:parabolic_deepnest}
  above are immediate. Further, since
  $N_{\epsilon}(\hat{W}(p)) \subset N_{5\epsilon}(K_p)$ and we assume
  $\epsilon < D_\Pi/5$, we know that the closed neighborhood
  $\overline{N}_{\epsilon}(\hat{W}(p))$ does not contain $p$, so
  Condition~\eqref{lemitm:parabolic_closure} holds as well.

  The stabilizer of a parabolic point $p \in \bgamp$ acts
  properly discontinuously on $\bgamp \minus \{p\}$. So, for any
  neighborhood $U$ of $p$ in $\bgamp$, there is a finite set $F$
  (depending on $U$) so that for any $\alpha \in \Gamma_p \minus F$, we have
  $\alpha \overline{\hat{W}(p)} \subset U$. Thus, by taking
  $U = g^{-1}B_{\epsilon/2}(q)$, we see that for $F_q = gF$,
  we have
  $\alpha \overline{\hat{W}(p)} \subset B_{\epsilon/2}(q)$ for all $\alpha \in g\Gamma_p \minus F_q$.

  We now define $W(q)$ and $V(q)$ exactly as stated in Conditions~\eqref{lemitm:parabolics_in} and~\eqref{lemitm:parabolic_coding_elt}. Our choice of $F_q$ ensures that
  $W(q) \subset B_{\epsilon/2}(q)$, so we know Condition~\eqref{lemitm:parabolic_small_diam} holds.

  It remains to verify that $W(q)$ and $V(q)$ are open
  neighborhoods of $q$. For this, it is enough to show that $W(q)$ and
  $V(q)$ each contain an open neighborhood of $q$. However, since
  $g\Gamma_p \cdot K_p = \bgamp \minus \{q\}$, and $\hat{V}(p)$ and
  $\hat{W}(p)$ both contain $K_p$, we know that $V(q)$ and $W(q)$ both
  contain the complement of the set $\bigcup_{\beta \in F_q} \beta
  K_p$. This set is a finite union of compact subsets not containing
  $q$, so its complement contains an open neighborhood of $q$.
\end{proof}

\subsection{Fixing $\bigepsilon$ and the geometric automaton} \label{subsec:fixing}

The construction of the automaton depends on a pair of nested finite
open covers of $\bgamp$, coming from the sets $V(z)$ and $W(z)$
provided by Lemma~\ref{lem:goodneighborhoods} and
Lemma~\ref{lem:goodneighborhoods_p}. The covers we construct via these
lemmas depend on the target neighborhood of the identity $\mathcal{V}$
needed in Theorem \ref{thm:main}.  Thus, we suppose now that we have
been given some $\mc{V} \subset C(\bgamp)$, and make the following
definition.

\begin{definition} \label{def:epsilon} 
  We fix a constant $\bigepsilon > 0$ so that $\bigepsilon$ is smaller than
  $\min(D/5, D_\Pi/5)$, and additionally such that every continuous
  self-map $\bgamp\to \bgamp$ $\bigepsilon$--close to the identity lies
  in $\mc{V}$.
\end{definition}

For each conical limit point $z \in \bgamp$, choose a pair of open
neighborhoods $V(z)^\ast\subset W(z)$ of $z$ and an element
$\alpha(z)$ as in the statement of Lemma~\ref{lem:goodneighborhoods},
for our chosen $\bigepsilon$.  Later, we will modify the $V(z)^\ast$
slightly, hence the provisional $\ast$ in the notation, rather than
simply calling the set $V(z)$ as in the Lemma statement.

For each parabolic point $p \in \Pi$ choose open sets
$\hat{V}(p)^\ast \subset \hat{W}(p)$ as in
Lemma~\ref{lem:goodneighborhoods_p}, again using the fixed
$\bigepsilon$. In addition, for each parabolic point $q \in \bgamp$ with
$q = gp$ for $p \in \Pi$, choose neighborhoods
$V(q)^\ast \subset W(q)$ of $q$ and a finite set
$F_q \subset g\Gamma_p$ as in the same lemma.

Let $Z \subset \bgamp$ be a finite collection so that the sets
$\{V(z)^\ast \}_{z\in Z}$ cover $\bgamp$.  We streamline notation as follows.  

\begin{notation} \label{notation:hats}
  For each conical limit point $z \in Z$, we denote $\alpha_z^{-1}V(z)^\ast$ by
  $\hat{V}(z)^\ast$ and $\alpha_z^{-1}(W(z))$ by $\hat{W}(z)$.  These sets
  will play an important role.

  If $q = gp$ for $p \in \Pi$, we will define
  $\hat{V}(q)^\ast := \hat{V}(p)^\ast$.
\end{notation}

We will set up an automaton with vertex set $Z$, and 
edges determined by the combinatorics of
the intersections these sets and labeled by certain elements of $\Gamma$.  For this we use the following definition.  

\begin{definition}
  For each point $z \in Z$, let
  $L(z) \subset \Gamma$ be defined as follows.  
  \begin{itemize}
  \item If $z$ is a conical limit point, then $L(z)$ is the singleton
    $\{\alpha_z\}$.
  \item If $z$ is a parabolic point $q = gp$ for $p \in \Pi$, then we
    define $L(z) := g\Gamma_p - F_q$, where $F_q$ is the set given by Lemma \ref{lem:goodneighborhoods_p}.  
  \end{itemize}
\end{definition}

Since we ultimately want to use this automaton to prove stability
properties of the action of $\Gamma$ on $\bgamp$, we need to make sure
that this intersection pattern is stable. More precisely, we want to modify the sets $V(z)^\ast$ and $\hat{V}(z)^\ast$ slightly to sets $V(z)$ and $\hat{V}(z)$ which satisfy
$\overline{\hat{V}(z)} \cap \overline{V(y)} = \emptyset$ if and only
  if $\hat{V}(z) \cap V(y) = \emptyset$.
  
For this we use the following lemma.  
  \begin{lemma}
  \label{lem:parabolic_nbhds_continuous}
  For any $z \in Z$ and any $\eta > 0$, there exists $r > 0$ so that
  the set
  \[
    \bigcup_{\alpha \in L(z)} \alpha N_r(\hat{V}(z)^\ast)
  \]
  is contained in the $\eta$-neighborhood of $V(z)^\ast$.
\end{lemma}
\begin{proof}
  When $z$ is a conical limit point, then $L(z)$ is a singleton
  $\{\alpha_z\}$ and $V(z)^\ast = \alpha_z \hat{V}(z)^\ast$. 
  Thus, $\alpha_z N_r(\hat{V}(z)^\ast) \subset N_\eta( \alpha_z \hat{V}(z)^\ast)$ holds for all 
  sufficiently small $r$ because $\Gamma$ acts by homeomorphisms.
  Now assume that
  $z$ is parabolic point $q = gp$ for $p \in \Pi$, and
  $L(z) = g\Gamma_p \minus F_q$.

  Part \ref{lemitm:parabolic_closure} of
  Lemma~\ref{lem:goodneighborhoods_p} ensures that we can choose $r$
  small enough so that the closure of $N_r(\hat{V}(p)^\ast)$ does not
  contain $p$. Then, for a finite subset $E_\eta \subset g\Gamma_p$,
  every $\alpha \in g\Gamma_p \minus E_\eta$ satisfies
  $\alpha N_r(\hat{V}(p)) \subset B_\eta(q)$. As $V(q)$ contains $q$,
  this ensures that
  $\bigcup_{\alpha \in g\Gamma_p \minus E_\eta} N_r(\hat{V}(p)^\ast)$
  is contained in $N_\eta(V(q)^\ast)$.

  Since the set $E_\eta \minus F_q$ is finite, we can choose
  $r$ small enough so that $\alpha N_r(\hat{V}(p)^\ast)$ is contained in
  $N_\eta(\alpha \hat{V}(p)^\ast)$ for every
  $\alpha \in E_\eta \minus F_q$. Since $V(q)$ contains the union of
  sets $\alpha \hat{V}(p)^\ast$ for $\alpha \in g\Gamma_p \minus F_q$, this
  guarantees that the desired inclusion holds.
\end{proof}

The following proposition collects the key properties of the combinatorics of an open cover that we will use to build the automaton. 

\begin{proposition} \label{prop:C_properties} There exist sets
  $V(z) \supset V(z)^\ast$ and $\hat{V}(z) \supset \hat{V}(z)^\ast$ so
  that for each $z \in Z \setminus \Pi$ we have
  $\hat{V}(z) = \alpha_z^{-1}V(z)$, and for every $z \in Z$, we have:
  \begin{enumerate}[label=(C\arabic*)]
  \item\label{item:smalldiam} $\diam(W(z)) < \bigepsilon$;
  \item\label{item:bigW} $\diam(\hat{W}(z)) > 4\bigepsilon$;
  \item\label{item:preimagenest}
    $\overline{N}_{2\bigepsilon}(\hat{V}(z)) \subset
    \hat{W}(z)$ 
  \item \label{item:edgenest} The set $W(z)$ is equal to
    $\{z\} \cup \bigcup_{\alpha \in L(z)} \alpha \hat{W}(z)$;
  \item \label{item:edgedefinition}The set $V(z)$ is equal to
    $\{z\} \cup \bigcup_{\alpha \in L(z)}\alpha \hat{V}(z)$ and $\overline{V(z)} \subset W(z)$. 
    \item
  \label{item:stable_intersections} For any pair $y, z \in Z$, we have
  $\overline{\hat{V}(z)} \cap \overline{V(y)} = \emptyset$ iff $\hat{V}(z) \cap V(y) = \emptyset$.
\end{enumerate}
\end{proposition}
Note that $\{V(z)\}_{z \in Z}$ is still a cover of $\bgamp$, and the index set $Z$, elements $\alpha_z$, and sets $W(z), \hat{W}(z)$ and $L(z)$ remain unchanged.

\begin{proof}
That the sets $V(z)^\ast$ and $\hat{V}(z)^\ast$ already chosen satisfy conditions \ref{item:smalldiam}-\ref{item:edgedefinition} nearly follows from Lemma \ref{lem:goodneighborhoods} and \ref{lem:goodneighborhoods_p}.  For the last part of \ref{item:edgedefinition} in the case of parabolic $z$ we must also use Lemma~\ref{lem:parabolic_nbhds_continuous}.

We can replace each of the sets
$\hat{V}(z)^\ast$ with a slightly larger set
$\hat{V}(z)$, and define $V(z) = \alpha_z(\hat{V}(z))$ for conical $z$, and via the expression in \ref{item:edgedefinition} for parabolic $z$.  
Lemma \ref{lem:parabolic_nbhds_continuous} tells us that we can do this so that
$V(z)$, $\hat{V}(z)$ are
respectively contained in arbitrarily small neighborhoods of $V(z)^\ast$
and $\hat{V}(z)^\ast$, and property~\ref{item:stable_intersections} above
holds for $V(z), \hat{V}(z)$.

By choosing $V(z)$ and
$\hat{V}(z)$ sufficiently close to $V(z)^\ast$, $\hat{V}(z)^\ast$, we can
ensure that all of the properties
\ref{item:smalldiam}-\ref{item:edgedefinition} still hold after we
replace $V(z)^\ast$, $\hat{V}(z)^\ast$ with $V(z)$, $\hat{V}(z)$.
(Note that Condition~\ref{item:preimagenest} is open, since it is about the \emph{closure} of $N_{2\bigepsilon}(\hat{V}(z))$.  In particular the condition is preserved when $\hat{V}(z)$ is enlarged slightly.)
\end{proof}

\putinbox{For the rest of the paper, we fix the index set $Z$, the
  open cover $V(z)$, as well as open sets $\hat{V}(z)$, $W(z)$,
  $\hat{W}(z)$, and the sets $L(z)$ for each $z \in Z$, and assume
  that these sets satisfy properties
  \ref{item:smalldiam}-\ref{item:stable_intersections} above.}

\begin{definition}[The automaton]\label{def:automaton}
  Let $\mc{G}$ be the directed graph with vertex set $Z$.  For
  $(z, y) \in Z \times Z$, if $\hat{V}(z) \cap V(y) = \emptyset$ there
  are no directed edges from $z$ to $y$, and if
  $\hat{V}(z) \cap V(y) \ne \emptyset$, then for each
  $\alpha \in L(z)$ we put a directed edge from $z$ to $y$ labeled by
  $\alpha$.  Thus, the set of edges from $z$ to $y$ is either empty or
  in bijective correspondence with $L(z)$. See
  Figure~\ref{fig:conical_edge} and Figure~\ref{fig:parabolic_edge}.
\end{definition}

Note that there are infinitely many outgoing edges from $z \in Z$ if
and only if $z$ is parabolic. We have constructed our automaton so
that it satisfies the following key property:
\begin{proposition}
  \label{prop:proper_nesting}
  If there is an edge from $z$ to $y$ in $\mc{G}$ labeled by a group
  element $\alpha \in L(z)$, then we have the inclusions
  \begin{equation}
    \label{eq:edge_inclusions}
    \alpha \overline{N}_\bigepsilon(W(y)) \subsetneq \alpha\hat{W}(z) \subset W(z). \tag{E1}
  \end{equation} 
  \end{proposition}

\begin{proof}
  If there is an edge from $z$ to $y$, then $\hat{V}(z) \cap V(y)$ is
  nonempty.  By \ref{item:edgedefinition}, we have $V(y) \subset W(y)$, so 
  $W(y) \cap \hat{V}(z)$ is nonempty. 
  By \ref{item:preimagenest} we have 
  $\overline{N}_{2\bigepsilon}(\hat{V}(z)) \subset \hat{W}(z)$
  Since 
  $\diam(W(y)) < \bigepsilon$ by \ref{item:smalldiam},  we have
  $\overline{N}_{\bigepsilon}(W(y)) \subset \hat{W}(z)$.  Moreover, since
  the diameter of $W(y)$ is at most $\bigepsilon$, and the diameter of
  $\hat{W}(z)$ is at least $4\bigepsilon$ by \ref{item:bigW}, this
  inclusion is proper. This proves the left-hand inclusion above.

  Finally, property \ref{item:edgenest} implies that if
  $\alpha \in L(z)$, then $\alpha \hat{W}(z) \subset W(z)$, which
  gives us the right-hand inclusion as well.
\end{proof}

\begin{figure}[htbp]
  \centering
  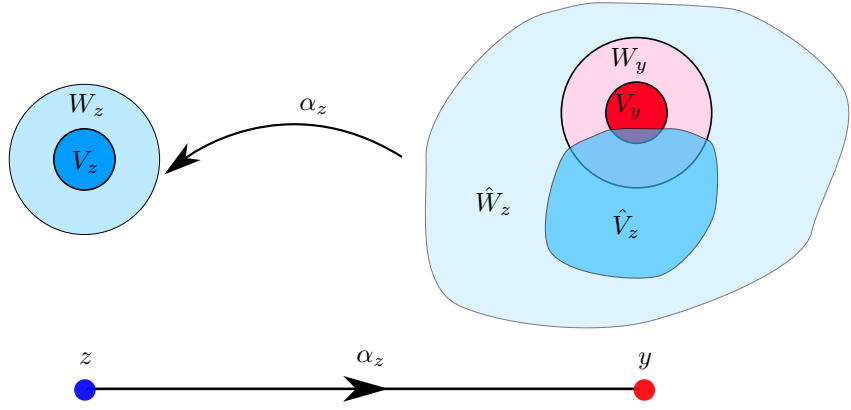
  \caption{If $z$ is conical, and $\hat{V}_z$ meets $V_y$, then there
    is an edge from $z$ to $y$ labeled $\alpha_z$. The group element
    $\alpha_z$ may or may not fix the point $z$.}
  \label{fig:conical_edge}
\end{figure}

\begin{figure}[htbp]
  \centering
  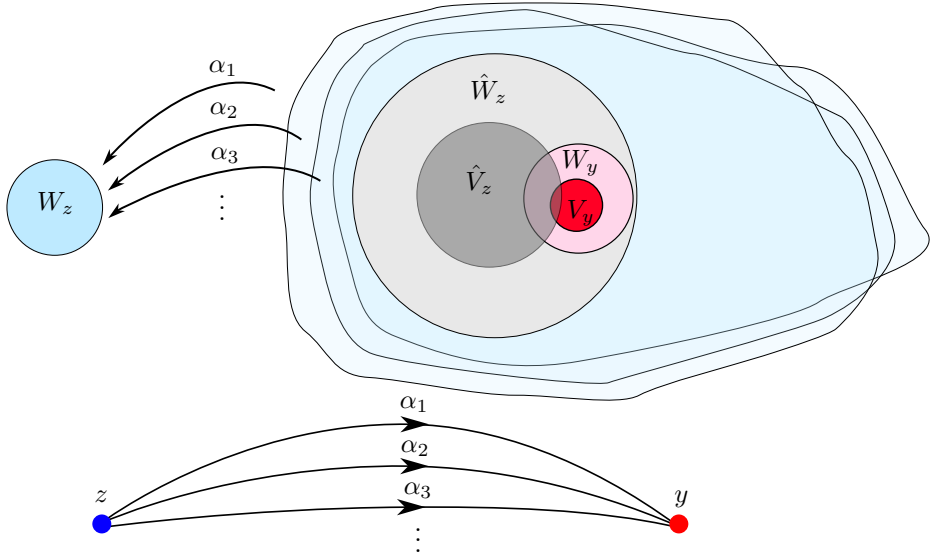
  \caption{If $z$ is parabolic, and $\hat{V}_z$ meets $V_y$, then for
    each $\alpha_i$ in $L(z)$ there is an edge from $z$ to $y$ labeled
    $\alpha_i$. In general $\alpha_i$ may or may not fix $z$, meaning
    that $W_z \cap \alpha_iW_z$ may or may not be empty.}
  \label{fig:parabolic_edge}
\end{figure}

We note the following consequence of Proposition \ref{prop:proper_nesting} for future use.  
\begin{remark} \label{rem:finitely_many} Consider a parabolic vertex
  $z = gp \in Z$, where $p \in \Pi$. If there is an edge from $z$ to
  another vertex $y \in Z$, by \eqref{eq:edge_inclusions} we have we
  have
  $\alpha \overline{N}_\bigepsilon(W(y)) \subsetneq \alpha \hat{W}(z)
  \subset W(z)$ for all $\alpha \in L(z)$, hence, for all but {\em
    finitely many} $\alpha \in g\Gamma_p$.  It follows that the closed
  neighborhood $\overline{N}_\bigepsilon(W(y))$ does not contain $p$.

  As a consequence, for any given $\eta > 0$ and any parabolic vertex
  $z \in Z$, if there is an edge from $z$ to $y$, then the inclusion
  $\alpha N_{\bigepsilon}(W(y)) \subset B_{\eta}(z)$ holds for all but
  finitely many $\alpha \in L(z)$. In particular, by choosing $\eta$
  sufficiently small, we can ensure that for all but finitely many
  exceptional $\alpha \in L(z)$, we have
  \[
    N_\eta(\alpha N_{\bigepsilon}(W(y))) \subset W(z).
  \]
  Further, since the edge inclusion condition
  \eqref{eq:edge_inclusions} still holds for the finitely many
  exceptional $\alpha$, there is some $\epsilon_z > 0$ so that for
  \emph{every} $\alpha \in L(z)$, we have
  \begin{equation} \label{eq:ez}
    N_{\epsilon_z}(\alpha(N_{\bigepsilon} W(y)))\subset W(z).
  \end{equation}
\end{remark}

%---------------------

\section{Properties of the $\rho_0$--automaton} \label{sec:automaton_prop}

In this section we explain how to associate points of $\bgamp$ to edge
paths in $\mc{G}$, called ``codings.''  While points of $\bgamp$ may
have more than one coding, we also show that any two codings of the
same point are geometrically related.

\begin{notation}
  We will often need to work with both finite-length and
  infinite-length edge paths in the graph $\mc{G}$; some of our
  results will apply to all edge paths, while others may apply only to
  infinite paths or only to finite paths.

  Whenever we refer to an edge path (or any other sequence) which may
  be either finite or infinite, we will let $I$ denote an index set
  for the sequence, which can be equal to either $\bN = \{1, \ldots\}$
  or $\{1, \ldots, n\}$ for some $n$, depending on context.
\end{notation}

\begin{definition}
  A \emph{strict conical $\mc{G}$--coding} is an infinite edge path in
  $\mc{G}$.  A \emph{strict parabolic $\mc{G}$--coding} is a finite
  edge path terminating in a parabolic
  point.  Note that it is possible that this edge path has length zero,
  in which case the coding is just a single parabolic point $z \in Z$.
  
  We use the notation $\head(e)$, $\tail(e)$, and $\lab(e)$ to denote
  (respectively) the initial vertex, terminal vertex, and label of an
  edge $e$ in the directed graph $\mc{G}$.  Thus, $\head(e)$ and $\tail(e)$ are elements of $\bgamp$ and $\lab(e) \in G$.  
   If
  $\coding{e} = (e_k)_{k \in I}$ is an edge path in $\mc{G}$, we call
  $(\tail(e_k))_{k \in I}$ the {\em terminal vertex sequence} and
  $(\lab(e_k))_{k \in I}$ the \emph{label sequence} for
  $\coding{e}$. We also define the \emph{initial vertex sequence} for
  $\coding{e}$ in the same way, except that for convenience we
  index this sequence starting from zero (so its $k$th term
  is $\head(e_{k+1})$).
  \end{definition}
  
\begin{definition} \label{def:coding} For a strict conical coding
  $\coding{e} = (e_k)_{k\in \bN}$ with label sequence
  $(\alpha_k)_{k \in \bN}$, and terminal vertex sequence
  $(z_{k})_{k \in \bN} = (\tail(e_k))_{k \in \bN}$, if
  \begin{equation*}
    \zeta \in \bigcap_{k = 0}^\infty\alpha_1 \cdots \alpha_k\overline{W(z_{k})}
  \end{equation*}
  we say that $\coding{e}$
  is a \emph{strict $\mc{G}$--coding of $\zeta$}.    If $\mc{G}$ is
  understood, we may omit it, and speak of a \emph{strict coding of $\zeta$}. 

  Similarly, a strict parabolic coding with label sequence $(\alpha_k)_{k\in\{1,\ldots,n\}}$ is a \emph{strict $\mc{G}$--coding of the parabolic point $\zeta$} if
  \begin{equation*}
    \zeta = \alpha_1\ldots \alpha_n q ,
  \end{equation*}
  where $q$ is the terminal point of the last edge.  If $\mc{G}$ is
  understood, we may speak simply of a sequence that (strictly)
  \emph{codes $\zeta$}.
\end{definition}

Ultimately we want to use $\mc{G}$-codings of points in $\bgamp$ to
understand perturbations of the $\Gamma$-action on $\bgamp$, so it is
useful to introduce a formalism which also allows $\Gamma$ to act on
the set of codings. In \cite{KKL}, this kind of idea is referred to as
``Sullivan's trick.''
\begin{definition}
  A \emph{generalized $\mc{G}$-coding} is a pair $(g_0, \coding{e})$,
  where $g_0$ is any element of $\Gamma$, and $\coding{e}$ is a strict
  $\mc{G}$-coding. The generalized coding $(g_0, \coding{e})$ is
  \emph{conical} if $\coding{e}$ is conical, and \emph{parabolic} if
  $\coding{e}$ is parabolic.

  If $\coding{e}$ codes a point $\xi \in \bgamp$, then we say that
  $(g_0, \coding{e})$ is a $\mc{G}$-coding of the point $g_0\xi$. We
  refer to the element $g_0$ as the \emph{initial point} of the coding
  $(g_0, \coding{e})$. Slightly abusing terminology, we also refer to
  the terminal vertex sequence of $\coding{e}$ as the \emph{terminal
    vertex sequence} for $(g_0, \coding{e})$, and similarly for label
  sequences and initial vertex sequences.

  Note that any strict coding can be viewed as a generalized coding by
  taking the initial point $g_0$ to be the identity. Occasionally we
  will refer to generalized $\mc{G}$-codings as ``$\mc{G}$-codings,''
  or even just ``codings'' if $\mc{G}$ is understood from
  context.
\end{definition}

\begin{lemma} \label{lem:codings_exist} Every point in $\bgamp$ has a
  strict $\mc{G}$-coding.
\end{lemma} 

\begin{proof} 
  Fix a point $\zeta \in \bgamp$.  Since the $V(z)$ cover $\bgamp$, we
  can choose $z_0$ so that $\zeta \in V(z_0)$.  If $\zeta = z_0$ and
  $z_0$ is parabolic, we stop.  The length-zero edge path consisting
  of the vertex $z_0$ is a strict parabolic coding of $\zeta$.

  Otherwise, we inductively define a strict $\mc{G}$-coding as
  follows. Assume that we have already defined a (possibly empty) edge
  path  $e_1, \ldots e_n$ starting at the point $z_0$ chosen in the last paragraph, and 
  with $\tail(e_k)=z_k$ for each $k>0$.
  Set
  $\alpha_k = \lab(e_k)$ and make the inductive hypothesis that
  \[
    \zeta_n := \alpha_n^{-1} \cdots \alpha_1^{-1}\zeta \in V(z_n).
  \]
By Proposition \ref{prop:C_properties}, we have $V(z_n)\subset W(z_n)$ and thus 
  \[
    \zeta \in \alpha_1 \cdots \alpha_n W(z_n).
  \]
  There are now two possibilities. If $z_n$ is not a parabolic
  point, then we let $\alpha_{n+1}$ be the only element in the
  singleton set $L(z_n)$. Then define
  $\zeta_{n+1} = \alpha_{n+1}^{-1}\zeta_n$, which lies in
  $\alpha_{n+1}^{-1}V(z_n) = \hat{V}(z_n)$. Since the sets $V(z)$ cover, 
  we can find some
  $z_{n+1} \in Z$ so that $\zeta_{n+1} \in V(z_{n+1})$.  Then there is
  an edge $e_{n+1}$ between $z_{n}$ and $z_{n+1}$ because
  $\hat{V}(z_{n})$ and $V(z_{n+1})$ have nonempty intersection.

  If instead $z_n$ is a parabolic point, then we may have that
  $\zeta = \alpha_1 \ldots \alpha_{n} z_{n}$, in which case we are
  done and have found a strict parabolic coding of $\zeta$.
  Otherwise, by property \ref{item:edgedefinition} of our open covers,
  there is some $\alpha_{n+1} \in L(z_{n})$ so that
  $\alpha_{n+1}^{-1}(\zeta_n) \in \hat{V}(z_{n})$.  Define
  $\zeta_{n+1} = \alpha_{n+1}^{-1}\zeta_n$, as above, pick some
  $z_{n+1}$ so that $\zeta_{n+1} \in V(z_{n+1})$.  There is an edge
  $e_{n+1}$ between $z_{n}$ and $z_{n+1}$ because $\hat{V}(z_{n})$ and
  $V(z_{n+1})$ have nonempty intersection.

Thus, if the inductive procedure terminates at a finite stage, we have produced a parabolic coding of $\zeta$.  Otherwise, by construction we produce an infinite edge path with labels $\alpha_k$ such that 
$\zeta \in \bigcap_n \alpha_1 \cdots \alpha_n \overline{W(z_{n})}$, as desired. 
\end{proof}

It turns out that conical $\mc{G}$-codings actually code (unique) {\em
  conical} points in $\bgamp$, although this is not obvious from the
lemma above. We will prove this fact later in Corollary
\ref{cor:conical}. We will also eventually show that conical codings
determine sequences in $\Gamma$ that are well-defined up to some
bounded error (Lemma \ref{lem:codings_bounded_dist}).

To prove these two facts, we first show that codings define sequences
of elements that stay close to geodesic rays in the cusped space
$X$. We use the following basic fact about hyperbolic metric spaces.
The statement is a rephrasing of Lemma 3.2 of \cite{MMW1}, and we
refer the reader there for a proof.
\begin{lemma}[See \cite{MMW1} Lemma 3.2]
  \label{lem:triangle_tripod}
  Let $X$ be a proper $\delta$-hyperbolic metric space, and fix a metric
  $d_\partial$ on the Gromov boundary $\partial X$ and a basepoint
  $x_0 \in X$. For any $\epsilon_0 > 0$ and any $R_1 > 0$, there
  exists a constant $R_2$ satisfying the following.

  Let $z_1, z_2, z_3$ be three points in $\partial X$, and for each
  $i,j$ distinct in $\{1,2,3\}$ let $[z_i, z_j]$ be a geodesic joining
  $z_i$ to $z_j$. If the distance between $z_i$ and $z_j$ is at least
  $\epsilon_0$ for each distinct pair $z_i, z_j$, then the
  intersection
  \[
    N_{R_1}([z_1, z_2]) \cap N_{R_1}([z_1, z_3])
  \]
  is contained in the $R_2$-neighborhood of a geodesic ray from $x_0$
  to $z_1$.
\end{lemma}

Now, using the strategy of Lemma 3.4 in \cite{MMW1}, we show the following. 

\begin{lemma} \label{lem:code_is_qg} There exists a uniform constant
  $R > 0$ so that, for any strict coding $\coding{e}$ of any point
  $\zeta \in \bgamp$, if $(\alpha_k)_{k \in I}$ is the associated
  label sequence, then the sequence
  \[
    g_k := \alpha_1 \cdots \alpha_k
  \]
  lies in the $R$-neighborhood of any geodesic ray in the cusped space
  $X$ based at the identity in $\Cay(\Gamma) \subset X$ and with
  endpoint $\zeta \in \partial X = \bgamp$. 
\end{lemma}

\begin{proof}
  Let $\epsilon_0 > 0$ be small enough so that
  $\epsilon_0 < \bigepsilon$, and so that for every $z \in \bgamp$, there
  exist points $z_+, z_- \in \bgamp$ so that
  \[
    \dvis(z, z_\pm) > \bigepsilon, \quad
    \dvis(z_+, z_-) > \epsilon_0.
  \]
  That such an $\epsilon_0$ exists follows from an easy geometric
  argument using the fact that $\bigepsilon < D/5$ (recall $D$ is the
  constant from Definition~\ref{def:pair_separation_constant}); see
  Lemma 3.1 in \cite{MMW1} for a proof.  We also choose a constant
  $R_1 > 0$ so that, for any pair of points $a, b \in \bgamp$ such
  that $\dvis(a, b) \ge \bigepsilon$, any geodesic in $X$ joining $a$ to
  $b$ passes within an $R_1$-neighborhood of the identity.
  
  Now fix $\zeta \in \bgamp$, and let $\coding{e}$
  be a strict $\mc{G}$-coding for $\zeta$, with label sequence
  $(\alpha_k)_{k \in I}$ and terminal vertex sequence
  $(z_k)_{k \in I}$. We define $z_0$ to be the initial vertex of the
  coding, and let $g_k := \alpha_1 \cdots \alpha_k$. For all
  $k \in \Izero$, we write $W_k$ for $W(z_k)$. Since
  $\coding{e}$ is a coding for $\zeta$ we know $\zeta \in W_0$.

  Choose points $\zeta_\pm \in \bgamp$ so that
  $\dvis(\zeta, \zeta_-) > \bigepsilon$,
  $\dvis(\zeta, \zeta_+) > \bigepsilon$, and
  $\dvis(\zeta_-, \zeta_+) > \epsilon_0$. Condition
  \ref{item:smalldiam} implies that the diameter of $W_0$ is less than
  $\bigepsilon$, so we know that $\zeta_+$ and $\zeta_-$ both lie in
  $\bgamp \minus W_0$. Let $[\zeta_-, \zeta]$ be a geodesic in $X$
  from $\zeta_-$ to $\zeta$ and $[\zeta_+, \zeta]$ a geodesic from
  $\zeta_+$ to $\zeta$.

  The edge inclusion condition \eqref{eq:edge_inclusions}, together
  with the fact that $\coding{e}$ is a strict coding for $\zeta$,
  implies that $g_k^{-1}\zeta \in W_k$ and
  $g_kN_{\bigepsilon}(W_k) \subset W_0$ for every $k$. Thus
  $g_k^{-1}(\bgamp \minus W_0)$ is a subset of
  $\bgamp \minus N_\bigepsilon(W_k)$, so $g_k^{-1}\zeta_-$ lies in
  $\bgamp \minus N_{\bigepsilon}(W_k)$ and therefore
  $\dvis(g_k^{-1}\zeta_-, g_k^{-1}\zeta) \ge \bigepsilon$.

  Thus, by our choice of $R_1$, the geodesic
  $g_k^{-1}[\zeta_-, \zeta]$ enters an $R_1$-neigh\-bor\-hood of the
  identity.  Equivalently, $g_k$ lies in the $R_1$-neighborhood of the
  geodesic $[\zeta_-, \zeta]$. The same argument also shows that $g_k$
  lies in an $R_1$-neighborhood of the geodesic $[\zeta_+,
  \zeta]$. Now apply Lemma \ref{lem:triangle_tripod} with
  $z_1 = \zeta$, $z_2 = \zeta_-$, and $z_3 = \zeta_+$ to see that
  there is a constant $R_2$ (independent of $\zeta$ and $\coding{e}$)
  so that $g_k$ lies in the $R_2$-neighborhood of some geodesic ray
  from the identity to $\zeta$.  Setting $R = R_2 + 2\delta$, we
  conclude that all $g_k$ lie in the $R$-neighborhood of any such
  geodesic ray.
\end{proof}

Although the previous lemma only applies directly to strict
codings, we can use it to obtain a statement for generalized codings
as well. Here and in what follows, we use the notation $|\alpha|_X$
for $\dcusp(\alpha, \idgamma)$, where $\idgamma$ is the image of the
identity element of $\Gamma$ in $\Cay(\Gamma) \subset X$.

\begin{corollary}
  \label{cor:generalized_code_is_qg}
  Let $R > 0$ be the constant from Lemma~\ref{lem:code_is_qg}, and let
  $(g_0, \coding{e})$ be a generalized $\mc{G}$-coding of a point
  $\zeta \in \bgamp$, with label sequence $(\alpha_k)_{k \in I}$. Then
  the sequence
  \[
    g_k := g_0 \cdot \alpha_1 \cdots \alpha_k
  \]
  lies in the $(R + |g_0|_X + 2\delta)$-neighborhood of any
  geodesic ray in $X$ from $\idgamma$ to $\zeta$.
\end{corollary}
\begin{proof}
  Lemma \ref{lem:code_is_qg} immediately implies that the sequence
  $(g_k)_{k \in \Izero}$ lies within an
  $R$-neighborhood of a geodesic $[g_0, \zeta]$ in $X$ from $g_0$ to
  $\zeta$. However, this geodesic lies within distance
  $|g_0|_X + 2\delta$ of any geodesic $[\idgamma, \zeta]$ in $X$ from
  $\idgamma$ to $\zeta$. To see this, simply make a long quadrilateral
  with one side a geodesic segment from $\idgamma$ to $g_0$, two sides
  long subsegments of the rays from $\idgamma$ and $g_0$ to $\zeta$,
  and the fourth side a short path between these rays far from
  $\idgamma$.  That quadrilaterals are $2 \delta$--slim now implies
  that $[g_0, \zeta]$ lies within the $R + 2\delta$ neighborhood of
  $[\idgamma, \zeta]$, and vice versa.
\end{proof}

The sequence $(g_k)_{k \in \Izero}$ from
Lemma~\ref{lem:code_is_qg} and
Corollary~\ref{cor:generalized_code_is_qg} will make many appearances,
so we make the following definition.

\begin{definition}
  If $(g_0, \coding{e})$ is a generalized coding with label sequence
  $(\alpha_k)_{k \in I}$, we call the sequence
  $(g_k)_{k \in \Izero}$ defined by
  $g_k := g_0 \cdot \alpha_1 \cdots \alpha_k$ the \emph{quasi-geodesic
    sequence associated to $(g_0, \coding{e})$}.
\end{definition}

\begin{remark}
The terminology ``quasi-geodesic sequence'' comes from the fact that
$(g_k)_{k \in \Izero}$ lies bounded distance from a geodesic in the
cusped space $X$, and behaves like a quasi-geodesic in the
\emph{relative Cayley graph} for $\Gamma$ (i.e. the Cayley graph for
$\Gamma$ defined with respect to the infinite generating set
$\mc{S} \cup \bigcup_{p \in \Pi} \Gamma_p$).  However, note that
$(g_k)_{k \in \Izero}$ may not be a quasi-geodesic in $X$, since the
distances $\dcusp(g_k, g_{k+1})$ may be arbitrarily large.  This complicating issue is the reason for much of the technical work in the next section.  

At this point, we have only shown (via Lemma~\ref{lem:code_is_qg})
that the associated quasi-geodesic sequence $(g_k)_{k \in \Izero}$
stays in a uniform neighborhood of a geodesic ray in $X$. To see that
the sequence actually follows the ray to infinity, we will use an
argument from Lemma 2.11 in \cite{MMW1}.
\end{remark}
\begin{lemma}[Bounded backtracking I]
  \label{lem:boundedbacktrack_I}
  Let $(g_0, \coding{e})$ be a generalized coding with terminal vertex
  sequence $(z_{k})_{k \in I}$ and associated quasi-geodesic sequence
  $(g_k)_{k \in \Izero}$. For every $k$, there is a proper inclusion
  \[
    g_{k+1}W(z_{k+1}) \subsetneq g_kW(z_{k}).
  \]
  Moreover, no element in $(g_k)_{k \in \Izero}$ is repeated more than
  $\# Z$ times.
\end{lemma}
\begin{proof}
  The first statement follows directly from
  Proposition~\ref{prop:proper_nesting}. Given this, we now prove the
  second statement. Suppose for a contradiction that for some
  $g \in \Gamma$ we have $\#\{k \st g_k = g\} > \#Z$. Then there are
  distinct $k, k' \in \bN$ such that $g_k = g_{k'} = g$ and
  $z_k = z_{k'}$. But then $g_kW(z_{k}) = g_{k'}W(z_{k'})$, which
  contradicts the proper inclusion already established.
\end{proof}

\begin{corollary} \label{cor:conical} Every conical
  $\mc{G}$--coding codes a unique conical limit point.
\end{corollary}
\begin{proof}
  The statement for generalized codings follows immediately from the
  statement for strict codings, so fix a strict conical coding
  $\coding{e}$.  That \emph{some} point $\zeta$ is coded by
  $\coding{e}$ follows immediately from Definition
  \ref{def:coding}. Lemma~\ref{lem:code_is_qg} and
  Lemma~\ref{lem:boundedbacktrack_I} imply that the point $\zeta$ is
  uniquely determined, because the associated quasi-geodesic sequence
  $(g_k)_{k \in \bN}$ tends to infinity in $\Gamma$ and stays in a
  uniform neighborhood of a geodesic ray in $X$ with ideal endpoint
  $\zeta$.  But this is just another way of saying that $\zeta$ is a
  conical limit point (see e.g. \cite[Prop. A.2]{HealyHruska10} for
  the equivalence).

\end{proof}
  
  To clarify the difference between strict parabolic and conical codings, we also note the following.  
  \begin{corollary} \label{cor:para_coding}
  If $\zeta \in \bgamp$ is a parabolic point, then $\zeta$ admits a strictly parabolic $\mc{G}$-coding.  
  \end{corollary} 
  \begin{proof} 
  Let $\zeta$ be a parabolic point.  By Lemma \ref{lem:codings_exist}, there is a strict coding of $\zeta$.  By Corollary  \ref{cor:conical}, this coding must be finite.  By definition, finite strict codings are parabolic strict codings.  
  \end{proof}

\begin{corollary}[Bounded backtracking II]\label{cor:boundedbacktrack_II}
  Let $R$ be the constant from Lemma \ref{lem:code_is_qg}. Suppose
  $(g_0, \coding{e})$ is a generalized conical coding with associated
  quasi-geodesic sequence $(g_k)_{k \in \Izero}$.  For any $m \in \bN$
  and any $n > m$, we have
  $d_X(id, g_n) > d_X(id, g_m) - (3R + 2|g_0|_X +
  6\delta)$. 
\end{corollary}

\begin{proof}
  Let $\zeta$ be the point coded by $(g_0, \coding{e})$, and let
  $\sigma$ be a ray in $X$ from the identity to $\zeta$. By
  Corollary~\ref{cor:generalized_code_is_qg}, we have
  $d(g_k, \sigma) < R + |g_0|_X + 2\delta$ for all $k$. Fix indices
  $m < n$, and let $x_m$ be a point along $\sigma$ which is within
  distance $R + |g_0|_X + 2\delta$ of $g_m$.

  Consider the strict conical coding $\coding{e}'$ in $\mc{G}$ given
  by the edge path $(e_{k+m})_{k \in \bN}$. The label sequence for
  this coding is the tail of the label sequence for $\coding{e}$, so
  the associated quasi-geodesic is the sequence $(g_k')_{k \in \bN}$,
  where $g_k' := g_m^{-1}g_{k+m}$. It follows that $\coding{e}'$ codes
  the point $g_m^{-1}\zeta$, hence (by Lemma~\ref{lem:code_is_qg}) the
  sequence $g_k'$ stays within distance $R$ of a geodesic ray
  $\sigma'$ from $\idgamma$ to $g_m^{-1}\zeta$.

  Let $\sigma''$ be the sub-ray of $\sigma$ from $x_m$ to $\zeta$, so
  that $g_m^{-1}\sigma''$ is a ray from $g_m^{-1}x_m$ to
  $g_m^{-1}\zeta$. Since
  $d_X(x_m, g_m) = d_X(g_m^{-1}x_m, \idgamma) < R + |g_0|_X +
  2\delta$, the rays $\sigma'$ and $g_m^{-1}\sigma''$ have Hausdorff
  distance bounded by $R + |g_0|_X + 4\delta$ (see the argument in the
  proof of Corollary~\ref{cor:generalized_code_is_qg}).

  Now, let $x_n'$ be a point on the ray $\sigma'$ so that
  $g_{n - m}' = g_m^{-1}g_n$ lies within distance $R$ of $x_n'$. There
  is a point $x_n''$ on $\sigma''$ so that
  $d_X(g_m^{-1}x_n'', x_n') < R + |g_0|_X + 4\delta$, meaning
  $d_X(g_m^{-1}g_n, g_m^{-1}x_n'') = d_X(g_n, x_n'') \le 2R + |g_0|_X
  + 4\delta$.  Using the fact that $\sigma''$ is a sub-ray of $\sigma$
  based at $x_m$, we have
  \[
    d_X(\idgamma, g_n) + d_X(g_n, x_n'') \ge 
    d_X(\idgamma, x_n'') = d_X(\idgamma, x_m) + d_X(x_m, x_n'') \ge
    d_X(\idgamma, x_m) 
  \]
  We have seen that
  \[  d_X(\idgamma, x_m)  \ge d_X(\idgamma, g_m) - (R + |g_0|_X + 2\delta) \]
  and  
  \[ d_X(g_n, x_n'') \le 2R + |g_0|_X + 4\delta.
  \]
  Combining the above inequalities gives the desired bound
  $d_X(id, g_n) > d_X(id, g_m) - (3R + 2|g_0|_X +
  6\delta)$.
\end{proof} 

As a consequence of Corollary \ref{cor:boundedbacktrack_II}, whenever
some edge label $\alpha_i$ in a generalized coding $(g_0, \coding{e})$
satisfies $|\alpha_i |_X > 3R + 2|g_0|_X + 6\delta$, 
the associated
quasi-geodesic sequence makes positive progress along the ray it
tracks. This will be important in the following section.

The final lemma of this section shows that conical codings are
``unique up to bounded distance,'' as follows.

\begin{lemma} \label{lem:codings_bounded_dist} For any
  $g_0, h_0 \in \Gamma$, there exists a constant $D_0 > 0$
  satisfying the following. Suppose that $(g_0, \coding{e})$,
  $(h_0, \coding{f})$ are two generalized codings of a common conical
  point $\zeta$.  Then the Hausdorff distance between the sets
  $\{g_k \st k \geq 0\}$ and $\{h_k \st k \geq 0\}$ (with respect to
  the metric $d_X$) is at most $D_0$.
\end{lemma}

Before proving the lemma, we fix notation for some more data related
to the automaton $\mc{G}$ which will appear in both the proof below
and in several arguments in the following section. 

\begin{definition}
  \label{defn:coset_representative}
  For each parabolic vertex $q$ of our automaton, we choose an element
  $t_q \in \Gamma$ so that $t_q^{-1}q \in \Pi$; we make this choice so
  that $|t_q|_X$ is minimized.

  Then define the quantity $C$ by
  \[
    C = 2\delta + 6 +  \max\left(\{|t_q|_X \st q\in Z\mbox{
        parabolic}\}\cup\{|\alpha_z|_X \st z\in Z\mbox{
        conical}\}\right).   \]
  Recall that when $z \in Z$ is conical, $\alpha_z$ is the unique
  element in the label set $L(z)$. 
\end{definition}

\begin{proof}[Proof of Lemma~\ref{lem:codings_bounded_dist}]
  Fix a conical point $\zeta \in \bgamp$ and a geodesic ray $\sigma$
  in $X$ from $\idgamma$ to $\zeta$, and let $g_0, h_0$ be arbitrary
  elements of $\Gamma$. Consider a pair of generalized codings
  $(g_0, \coding{e})$ and $(h_0, \coding{f})$ of $\zeta$. Letting $R$
  be the constant from Lemma \ref{lem:code_is_qg}, and defining
  \[
    R_0 = R + \max\{|g_0|_X, |h_0|_X\} + 6\delta,
  \]
  Corollary~\ref{cor:generalized_code_is_qg} implies that the
  associated quasi-geodesic sequences for both $(g_0, \coding{e})$ and
  $(h_0, \coding{f})$ lie in the set
  \[
    N_{R_0}(\sigma) \cap \Cay(\Gamma).
  \]
  Let $(g_k)_{k \in \Izero}$ be the quasi-geodesic sequence associated
  to $(g_0, \coding{e})$, and let $(e_k)_{k \in \bN}$ and
  $(\alpha_k)_{k \in \bN}$ be the sequence of edges and labels for
  $\coding{e}$.

  Suppose that for some particular $k > 0$, we have
  $|\alpha_k|_X \ge C$. Then $\head(e_k)$ is parabolic,
  equal to $t_k p_k$ for some $p_k \in \Pi$ and $t_k \in \Gamma$
  chosen in Definition~\ref{defn:coset_representative}.  We thus have
  $\alpha_k = t_ka_k$ for $a_k \in \Gamma_{p_k}$, and the group
  elements $g_{k-1}t_k$ and $g_k = g_{k-1}t_ka_k$ lie on the boundary
  of a common horoball $\mc{H}$ in $X$.  We let $\tau_k$ be a regular
  $\mc{H}$--geodesic joining $g_{k-1}t_k$ to $g_k$.

  Let $\tau_k'$ be the subsegment of $\tau_k$ contained in the
  $\delta+1$--horoball nested inside $\mc{H}$.  Our assumption on
  $|\alpha_k|_X$ ensures that $\tau_k'$ is non-empty.
  Lemma~\ref{lem:gmlemma} says that $\tau_k'$ is an $X$--geodesic, and
  the remaining subsegments of $\tau_k$ are vertical, so they are also
  $X$--geodesics.  The endpoints $g_{k-1}t_k$ and $g_k$ of $\tau_k$
  are distance at most $R_0 + C$ from points $s_k$ and $s_{k+1}$ on
  $\sigma$.  We thus obtain a geodesic hexagon, one of whose sides is
  a part of $\sigma$, with the opposite side equal to $\tau_k'$.  Any
  point of $\tau_k \minus \tau_k'$ is at most $R_0 + C + \delta +1$
  from either $s_k$ or $s_{k+1}$.  And any point of $\tau_k'$ is
  within $4\delta$ of one of the other five sides of the hexagon,
  hence within $R_0 + C + 5\delta + 1$ of a point of $\sigma$.

  For each index $k$ such that $|\alpha_k|_X \ge C$, we fix a path
  $\tau_k$ as above.  Consider the set
  \begin{equation}\label{eq:Ye}
    Y_{\coding{e}} = \{\idgamma\} \cup \{g_j \st j \ge 0\} \cup \bigcup\{ \tau_k \st |\alpha_k|_X\ge C\}.
  \end{equation}
   
  Since each $g_j$ for $j \ge 0$ is contained in an $R_0$-neighborhood
  of $\sigma$, and each segment $\tau_k$ is contained within an
  $(R_0 + C + 5\delta + 1)$-neighborhood of $\sigma$, the whole set
  $Y_{\coding{e}}$ is also contained within an
  $(R_0 + C + 5\delta + 1)$-neighborhood of $\sigma$.  We also know
  that the set $Y_{\coding{e}}$ is $(C + |g_0|_X)$-coarsely connected,
  since $\dcusp(\idgamma, g_0) = |g_0|_X$ by definition, and for each
  $k > 0$ either $\dcusp(g_{k-1},g_k) = |\alpha_{k}|_X < C$ or there
  is a path $\tau_k$ in $Y_{\coding{e}}$ with one endpoint equal to
  $g_k$ and the other within $C$ of $g_{k-1}$.

  Lemma~\ref{lem:boundedbacktrack_I} implies that the set of points
  $\{g_j \st j \ge 0\}$ has unbounded diameter in $X$, which means
  that there are points of $X$ arbitrarily far along $\sigma$ that lie
  within distance $R_0 + C + 5\delta + 1$ of $Y_{\coding{e}}$.
  This means that $Y_{\coding{e}}$ is actually within Hausdorff
  distance $R'$ of $\sigma$, for a constant $R'$ depending only on
  $R_0, C, \delta$, and $|g_0|_X$.

  The coding $(h_0, \coding{f})$ has an associated set
  $Y_{\coding{f}}$ defined analogously to the way $Y_{\coding{e}}$ was
  defined above in~\eqref{eq:Ye}.  The same argument shows that the
  Hausdorff distance from $\sigma$ to $Y_{\coding{f}}$ is at most
  $R''$, for a constant $R''$ depending only on $R_0, C, \delta$, and
  $|h_0|_X$.  Thus the Hausdorff distance between $Y_{\coding{e}}$ and
  $Y_{\coding{f}}$ is at most $R' + R''$.

  Now consider some $h_k$ in the associated quasi-geodesic sequence
  for $(h_0, \coding{f})$. We wish to show that $h_k$ lies uniformly
  close to some point $g_j$ in the associated sequence for
  $(g_0, \coding{e})$.  We know that there is some point
  $p\in Y_{\coding{e}}$ so that $d_X(p,h_k)\le R' + R''$.  If
  $p = g_j$ for some $j$ there is nothing left to show. If
  $p = \idgamma$, then
  $\dcusp(g_0, h_k) \le \dcusp(g_0, \idgamma) + \dcusp(\idgamma, h_k)
  \le R' + R'' + |g_0|_X.$ Finally, if $p \in \tau_j$ for some
  $j \ge 0$, we note that $p$ has depth at most $R' + R''$, so it is
  distance at most $R' + R'' + 2$ from an endpoint of $\tau_j$, and
  hence at most $R' + R'' + C + 2$ from either $g_{j-1}$ or $g_j$.  In
  each case, we have shown that $h_k$ lies in the
  $(R' + R''+ |g_0|_X + C + 2)$--neighborhood of the associated
  quasi-geodesic sequence for $(g_0, \coding{e})$; we can then argue
  symmetrically to obtain the desired uniform bound on Hausdorff
  distance.
\end{proof}

%---------------------

\section{Uniform nesting}\label{sec:uniform_nest}
In this section, we prove an analog of \cite[Lemma 3.8]{MMW1} (there
called the Uniform Contraction Lemma).  More care is needed here
because of the presence of parabolic elements. 

We have seen that any two generalized codings of the same conical
point have infinitely many nearby pairs of points along their
quasi-geodesic sequences $g_k$ and $h_k$.  The next condition says,
roughly, that there are two possibilities.  One possibility is that
these sequences really look quasi-geodesic: sequential points
$g_j, g_{j+1},...$ are eventually uniformly spaced, and the same holds
for the points in the sequence $h_k$.  The second possibility is that
there is an infinite sequence of large parabolic jumps between
successive $g_j, g_{j+1}$ (and the same holds for nearby points in the
sequence $h_k$).  In either case, one can control the behavior of a
sequence of nested sets determined by one coding in terms of the other
coding, which is what we will need to prove that codings determine a
well-defined semi-conjugacy between the standard action and a small
perturbation.

\begin{definition} \label{def:unif_nest} Let $(g_0, \coding{e})$ and
  $(h_0, \coding{f})$ be generalized codings of the same conical point
  in $\bgamp$. Let $g_k, h_k$ be their associated quasi-geodesic
  sequences, and let $W^\coding{e}(k) = W(\tail(e_k))$ and
  $W^\coding{f}(k) = W(\tail(f_k))$.  This pair of codings has the
  \emph{$c$-uniform nesting property} if for any $\smallepsilon<c$,
  there are constants
  $D_1 > 0$ and $D_2 = D_2(\smallepsilon)> 0$ 
  so that at least one of the following
  conditions holds. 
    \begin{enumerate}
    \item \emph{(Uniform nesting with short
        words)} \label{itm:shortword_nesting}  
       There exist $N, M \in \bN$, and sequences $n_k := k + M$ and $m_k$, such that 
       for all $k$ we have 
       \begin{itemize}
       \item 
      $|\lab(e_{n_k})|_X \le D_2$,  
      \item 
      $d_X(g_{n_k},h_{m_k})\le D_1$, 
      and 
      \item
       $g_{n_k + N} \overline{W^\coding{e}(n_k+N)}\subset h_{m_k}W^\coding{f}(m_k).$
        \end{itemize} 
        
    \item\label{itm:longparabolic} \emph{(Uniform nesting with long
        parabolics)} There are infinite sequences of indices $n_k$, $m_k$ such
      that for every $k \in \bN$, $d_X(g_{n_k},h_{m_k})\le D_1$ and 
      \begin{itemize}
      \item $\tail(e_{n_k})$ is a parabolic point;
      \item
        $g_{n_k} \overline{B}_{3\smallepsilon}(\tail(e_{n_k}))\subset
        h_{m_k} W^\coding{f}(m_k)$;
      \item
        $\lab(e_{n_k + 1})N_{\bigepsilon}(W^\coding{e}(n_k + 1) )\subset
        B_{\smallepsilon}(\tail(e_{n_k}))$, where $\bigepsilon$ is as in
        Definition \ref{def:epsilon}.
      \end{itemize}
    \end{enumerate}
  \end{definition} 
  
  See Remark \ref{rem:finite_nesting_condition} for an important interpretation of these conditions.  
  
Our main goal in this section is to prove:
\begin{proposition} \label{prop:unif_nest}
  Given $g_0$, $h_0 \in G$,
  there exists $c >0$ such that if $(g_0, \coding{e})$ and
  $(h_0, \coding{f})$ are generalized codings of the same conical
  point $\zeta \in \bgamp$, then $(g_0, \coding{e})$ and
  $(h_0, \coding{f})$ have the $c$-uniform nesting
    property.  Furthermore, the constant $D_1$ depends only on
  $g_0, h_0$, and the constants $N$ and $D_2$ depend only on $g_0$,
  $h_0$ and the choice of $\smallepsilon<c$, and not on $\zeta$,
  $\coding{e}$, or $\coding{f}$.
 \end{proposition}
 
As a first step towards the proof, the following lemma describes the behavior of codings that involve edges labeled by long words in $\Gamma$.  Since codings are close to geodesics, these long words correspond to parabolics -- geometrically, the geodesic has a long segment through the horosphere based at this parabolic point.  The lemma makes precise the notion that two codings of the same point, being close to the same geodesic, have long segments in common horospheres.   

\begin{lemma}[Large jumps come from common
  parabolics] \label{lem:para_jump} Given $g_0, h_0 \in \Gamma$, there
  are constants $D_1, J > 0$ so the following holds.  Let
  $(g_0, \coding{e})$ and $(h_0, \coding{f})$ be two generalized
  codings of the same conical point $\zeta$, let
  $\lab(e_k) = \alpha_k$, $\lab(f_k) = \beta_k$, and let
  $(g_k)_{k \in \bN \cup \{0\}}, (h_k)_{k \in \bN \cup \{0\}}$ be the
  respective associated quasi-geodesic sequences.  Then for each $n$
  with $|\alpha_n|_X > J$
  there exists $m = m(n)$ such that 
\begin{enumerate} 
\item $\dcusp(g_{n-1}, h_{m-1})<D_1$ and $\dcusp(g_n, h_m) < D_1$;
\item There exists $p \in \Pi$ and $g,h\in G$ such that
  $\head(e_n) = gp$ and $\head(f_m) = hp$; 
\item $g_{n-1}g$, $h_{m-1}h$, $g_n$ and $h_m$ are all in the same coset of $\Gamma_p$.  
\end{enumerate}
\end{lemma}

The rough idea of the proof is as follows. We consider a regular
geodesic ray $\sigma$ in $X$ from $\idgamma$ to $\zeta$, close to the quasi-geodesic sequences associated to both
generalized codings.
We first prove that, if the distance in $X$ between two consecutive
points $g_{n-1}, g_n$ is large, then a geodesic segment between
$g_{n-1}$ and $g_n$ must spend a large amount of its lifetime in some
horoball $\mc{H}$ (this will follow from Lemma~\ref{lem:tauH}
below). Using the fact that the quasi-geodesic sequence tracks $\sigma$, 
we prove that $\sigma$ must also spend some large
amount of time in the same horoball $\mc{H}$. If this time is long
enough, we can even conclude that $\sigma$ spends some large time in a
horoball $\mc{H}'$ nested deeply inside of $\mc{H}$.

Then, using the fact that the quasi-geodesic sequence associated to 
$\coding{f}$ also tracks $\sigma$, we
can show that there are points $h_{m-1}, h_m$ in this sequence 
on ``either side'' of
$\mc{H}'$. Since these points are far apart, a regular geodesic
joining $h_{m-1}$ to $h_m$ must spend a long amount of time in
$\mc{H}$, which essentially proves the lemma. 

To make the above reasoning precise, we first prove two general lemmas
about the geometry of the cusped space $X$. The first lemma gives us a
way to estimate the length of time a geodesic in $X$ spends inside a
horoball $\mc{H}$ when one of the endpoints of the geodesic lies on
the boundary of $\mc{H}$.
\begin{lemma}\label{lem:tauH}
  Suppose that $\alpha \in G$ lies in a coset $g\Gamma_p$ for some
  $g \in \Gamma$ and $p \in \Pi$, and let $\mc{H} \subset X$ be the
  horoball in $X$ based at $g\Gamma_p$. Let $\tau$ be a regular
  geodesic in $X$ from $\idgamma$ to $\alpha$. Then $\tau \cap \mc{H}$
  contains a subsegment with length at least
  $|\alpha|_X - (|g|_X + 12\delta)$.
  
  Moreover, if $\mc{H}' \neq \mc{H}$ is any other horoball in $X$,
  then any component of $\tau \cap \mc{H}'$ has length at most
  $|g|_X + 12\delta$. 
\end{lemma}
\begin{proof}
  Note that the elements $g$ and $\alpha$ both lie on the boundary of
  the horoball $\mc{H}$. Let $\check{\mc{H}}$ be the
  $3\delta$--horoball nested inside $\mc{H}$. (See Section
  \ref{sec:setup} for definitions.)  Let $a$ be the point on
  $\partial\check{\mc{H}}$ which is connected by a vertical path to
  $g$, and let $b$ be the point on $\partial\check{\mc{H}}$ which is
  connected by a vertical path to $\alpha$.

  Consider a geodesic pentagon whose vertices are (in cyclic order)
  $\idgamma$, $g$, $a$, $b$, and $\alpha$, so that the side connecting
  $\idgamma$ to $\alpha$ is $\tau$. The sides $[g,a]$ and $[\alpha,b]$
  are vertical, and we may suppose that the side $[a,b]$ is regular.
  
  Every point on $\tau$ is within $3\delta$ of some point on one of
  the other sides of this pentagon.  At most the initial subsegment of
  length $|g|_X+6\delta$ and the final subsegment of length $6\delta$
  can be close to some other side than $[a,b]$.

  We deduce that a subsegment $\tau'$ of length at least
  $|\alpha|_X - (|g|_X + 12\delta)$ is within $3\delta$ of $[a,b]$,
  and hence within $3\delta$ of $\check{\mc{H}}$.  In particular
  $\tau'$ is completely contained in $\mc{H}$, establishing the first part of the lemma.

  The last sentence follows since horoballs have disjoint interior.  
\end{proof}

The next lemma is essentially a consequence of the quasi-convexity of
horoballs in $X$.
\begin{lemma}[Nearby geodesics enter common horoballs]
  \label{lem:geodesics_common_horoballs}
  Let $\mc{H} \subset X$ be a $k$-horoball for some $k \ge 0$, and let
  $\tau_1, \tau_2$ be two regular geodesic segments in $X$ such that
  the endpoints of $\tau_1$ lie within distance $L$ of the endpoints
  of $\tau_2$. If $\tau_1 \cap \mc{H}$ contains a segment with length
  at least $T \ge 4L + 8\delta + 3$, then $\tau_2 \cap \mc{H}$
  contains a segment with length at least $T - (4L + 8\delta)$.
\end{lemma}
\begin{proof}
  Let $a_1, b_1$ and $a_2, b_2$ be the endpoints of $\tau_1, \tau_2$
  respectively, so that $d_X(a_1, a_2) \le L$ and
  $d_X(b_1, b_2) \le L$. We let $\check{\mc{H}}$ be the
  $(k + 2\delta)$-horoball nested inside of the $k$-horoball
  $\mc{H}$. Then, since $\tau_1$ is regular, and its intersection with
  $\mc{H}$ has length at least $T > 4\delta + 3$, it also intersects
  $\check{\mc{H}}$, and in fact the intersection contains a segment
  with length at least $T - 4\delta$.
  We let $x_1, y_1$ be
  the endpoints of such a segment, so that $a_1, x_1, y_1, b_1$ are
  arranged on $\tau_1$ in that order.

  The length of $\tau_1$ is
  \begin{align*}
    d_X(a_1,x_1) + d_X(x_1,y_1) + d_X(y_1,b_1) &  = d_X(a_1,x_1) + (T-4\delta) +d_X(y_1,b_1) \\
                                               & \ge d_X(a_1,x_1) + d_X(y_1,b_1) +4L + 4\delta + 3
  \end{align*}
  The geodesic $\tau_2$ has endpoints at most $L$ from those of $\tau_1$, so its length is at least $d_X(a_1,x_1) + d_X(y_1,b_1) +2L + 4\delta + 3$.  In particular, there are points  $x_2$ and $y_2$  of $\tau_2$ satisfying
  \begin{align*}
    d_X(a_2, x_2) = d_X(a_1, x_1) + L + 2\delta,\\
    d_X(y_2, b_2) = d_X(y_1, b_1) + L + 2\delta.
  \end{align*}
  Moreover, the points $a_2, x_2, y_2, b_2$ must lie on the segment $\tau_2$
  in that order, since $2(L + 2\delta) < 2L + 4\delta + 3$.
  
  Consider a geodesic quadrilateral with opposite sides
  $\tau_1, \tau_2$. The subsegment $[x_2, y_2] \subset \tau_2$ is
  contained in a $2\delta$-neighborhood of the other three sides of
  this quadrilateral. In fact, $[x_2, y_2]$ must be contained in the
  $2\delta$-neighborhood of $[x_1, y_1]$, meaning it is contained in
  the $2\delta$-neighborhood of $\check{\mc{H}}$ and therefore in
  $\mc{H}$. Thus, this segment in $\tau_2 \cap \mc{H}$ has length at
  least
  \[
    |\tau_2| - 2L - 4\delta - d_X(a_1, x_1) - d_X(y_1, b_1).
  \]
  Then, since $|\tau_2| \ge |\tau_1| - 2L$, by our choice of
  $x_1, y_1$ we have
  \[
    |\tau_1| \ge d_X(a_1, x_1) + (T - 4\delta) + d_X(y_1, b_1),
  \]
  so we obtain the desired bound.
\end{proof}

\begin{proof}[Proof of Lemma \ref{lem:para_jump}]
  Let $C$ be the constant from
  Definition~\ref{defn:coset_representative}, which we recall is
  (partly) determined by fixing a group element $t_q \in \Gamma$ for
  each parabolic vertex $q \in Z$, satisfying $t_q^{-1}q \in
  \Pi$. Letting $R$ be the constant from Lemma~\ref{lem:code_is_qg},
  we set
  \[
    R' = R + \max\{|g_0|_X, |h_0|_X\} + 2\delta,
  \]
  and then let
  \[
    J > 11 R' + 2 C + 50\delta.
  \]

  Consider two conical codings $(g_0, \coding{e})$ and
  $(h_0, \coding{f})$ of a point $\zeta \in \bgamp$, and let
  $(g_n)_{n \in \bN \cup \{0\}}$ and $(h_m)_{m \in \bN \cup \{0\}}$ be the associated
  quasi-geodesic sequences. 
  Fix some $n \in \bN$ such that $|\alpha_n|_X > J$.  Since $J>C$ the edge
  labeled by $\alpha_n$ has its initial vertex at a parabolic point, meaning $\alpha_n$ is in some coset $t_q\Gamma_p$ with
  $q = \head(e_n) = t_q p$ and $p\in \Pi$.  In particular
  $\alpha_n = t_q a$ where $a\in \Gamma_p$.  Let $\mc{H}$ be the
  horoball associated to the coset
  $g_n \Gamma_p = g_{n-1}t_q\Gamma_p$, and let $\tau$ be a regular
  geodesic from $g_{n-1}$ to $g_n$. 
  Now, $g_{n-1}^{-1}\tau$ is a regular geodesic from $\idgamma$ to $\alpha_n = g_{n-1}^{-1}g_n$, so by 
  Lemma~\ref{lem:tauH}, it contains a segment in $g_{n-1}^{-1}\mc{H}$ of length at least $J - (C+12\delta)$.  Thus, 
  $\tau\cap \mc{H}$ contains a segment of length at least
  $J - (C+12\delta)$.

  Let $\sigma$ be a regular geodesic from $\idgamma$ to $\zeta$. 
  For a pair of points $a, b \in \sigma$, let $[a,b]$ denote
  the subsegment of $\sigma$ with endpoints $a,b$. By
  Corollary~\ref{cor:generalized_code_is_qg}, we can find a sequence
  $(x_i)_{i \in \bN}$ on $\sigma$, so that 
  $d_X(g_i, x_i) \le R'$ for all $i$. Then
  Lemma~\ref{lem:geodesics_common_horoballs} implies that the
  intersection $[x_{n-1}, x_n] \cap \mc{H}$ contains a segment with
  length at least $J - C - 4R' - 20\delta$. 

  We now wish to find consecutive
  $h_{m-1}, h_m$ in the quasi-geodesic sequence associated to $(h_0, \coding{f})$
  so that a regular geodesic between $h_{m-1}$ and $h_m$ spends all
  but a uniformly bounded amount of its length inside the horoball
  $\mc{H}$. For this, we let $\mc{H}'$ be the $R'$-horoball nested
  inside of $\mc{H}$. 
  Since $R' > 2\delta$ by definition,
  Lemma~\ref{lem:gmlemma} implies that $\mc{H}'$ is convex, so the
  intersection $\sigma \cap \mc{H}'$ is a geodesic segment, which must
  have length at least $J - C - 6R' - 20\delta$. 
  In particular, the length of this segment is greater than $4R' + 8\delta + 3$, so we will be able to apply Lemma~\ref{lem:geodesics_common_horoballs} to it.
  Let $\sigma_1$ be the connected component of $\sigma \setminus \mc{H}'$ containing $\idgamma$ and let $\sigma_2$ be the unbounded connected component, which exists because $\zeta$ is conical. 

  Since every element in our sequence $h_m$ lies in $\Gamma$, no $h_m$
  can lie in an $R'$-neighborhood of the $R'$-horoball $\mc{H}'$, and
  thus each $h_m$ lies within distance $R'$ of exactly one of
  $\sigma_1$ or $\sigma_2$. We let $m$ be the first index so that
  $h_m$ is within distance $R'$ of $\sigma_2$, so $h_{m-1}$ lies distance at most $R'$ from $\sigma_1$. 
  Let
  $y_{m-1}, y_m$ be points on $\sigma$ which are within distance $R'$ of
  $h_{m-1}$ and $h_m$ respectively. These points must lie on either side of the
  intersection $\mc{H}' \cap \sigma$, so the segment 
  $[y_{m-1}, y_m] \cap \mc{H}'$ is a geodesic segment with length at
  least $J - C - 6R' - 20\delta$.
  Then, if $\tau'$ is a regular
  geodesic between $h_{m-1}$ and $h_m$,
  Lemma~\ref{lem:geodesics_common_horoballs} tells us that
  $\tau' \cap \mc{H}'$ contains a segment of length at least
  $J - C - 10R' - 28\delta > C$.

  In particular, we have
  $d_X(h_{m-1}, h_m) = |\beta_m|_X = |\tau'| > C$, so the initial
  vertex of the edge labeled by $\beta_m$ is a parabolic point
  $q' = t_{q'}p' \in Z$ for some $p' \in \Pi$. Lemma~\ref{lem:tauH}
  implies that the horoball based on $t_{q'}p'$ is the only horoball
  which $\tau'$ can intersect in a segment of length longer than
  $C + 12\delta$.
  Since $\tau'$ meets $\mc{H}$ in a segment longer than
  this, we must have $q = q'$ and $p' = p$, establishing the last two
  items of the lemma.

  It remains to bound $\dcusp(g_{n-1},h_{m-1})$ and
  $\dcusp(g_n, h_m)$. We first show that $\tau$ and $\tau'$ cross $\mc{H}$ in the same direction. 
   That is,
  we show:
  \begin{claim}
    $g_{n-1}$ is in the $R'$-neighborhood of the initial segment
    $\sigma_1$ of $\sigma$, and $g_n$ is in the $R'$-neighborhood of
    the final ray $\sigma_2$.
  \end{claim}
  \begin{proof}
    The claim essentially follows from the fact that the the distance $d_X(g_{n-1}, g_n)$ is large
     and from the bounded
    backtracking property of the quasi-geodesic sequence
    $(g_n)_{n \in \bN}$ (Corollary~\ref{cor:boundedbacktrack_II}). To
    be specific, we consider the points $x_{n-1}, x_n$ on $\sigma$
    which are within distance $R'$ of $g_{n-1}, g_n$
    respectively. Recall that the regular geodesic $\tau$ between
    $g_{n-1}$ and $g_n$ intersects $\mc{H}$ in a segment of length at
    least $J - C - 12\delta$, so $\tau \cap \mc{H}'$ contains a
    segment of length at least $J - C - 2R' - 12\delta$. Then
    Lemma~\ref{lem:geodesics_common_horoballs} implies that
    $[x_{n-1}, x_n] \cap \mc{H}'$ contains a segment of length at
    least $J - C - 6R' - 20\delta > 0$. So, either
    $x_{n-1} \in \sigma_1$ and $x_n \in \sigma_2$, or vice versa.

    However, if $x_{n-1} \in \sigma_2$ and $x_n \in \sigma_1$, then
    \[
      d_X(\idgamma, x_{n-1}) = d_X(\idgamma, x_n) + d_X(x_{n-1}, x_n),
    \]
    implying that
    \[
      d_X(\idgamma, x_n) \le d_X(\idgamma, x_{n-1}) - (J - C - 6R' -
      20\delta).
    \]
        Since
    $d_X(x_{n-1}, g_{n-1}) \le R'$ and $d_X(x_n, g_n) \le R'$, this
    would imply
    $|g_n|_X \le |g_{n-1}|_X - (J - C - 8R' - 20\delta) < |g_{n-1}|_X
    - (3R' + 2\delta)$.
    Since $R' \ge R + |g_0|_X + 2\delta$ by definition, this
    contradicts Corollary~\ref{cor:boundedbacktrack_II}. This proves
    the claim.  \end{proof}

  We now consider the points $x_{n-1}, y_{m-1}$ on $\sigma_1$ within
  distance $R'$ of $g_{n-1}, h_{m-1}$. Let $z_1, z_2$ be the
  endpoints of $\sigma_1$ and $\sigma_2$ on $\partial\mc{H}'$,
  respectively.

  Since $d_X(g_{n-1}, \mc{H}') \le C + R'$ and
  $d_X(h_{m-1}, \mc{H}') \le C + R'$, we have
  $d_X(x_{n-1}, \mc{H}') \le 2R' + C$ and
  $d_X(y_{m-1}, \mc{H}') \le 2R' + C$. Then, we apply
  Lemma~\ref{lem:quickapproach} to the subsegment $[x_{n-1}, z_2]$
  and the $R'$-horoball $\mc{H}'$ to deduce that
  $d_X(x_{n-1}, \partial\mc{H}')$ differs by at most $\delta$ from
  $d_X(x_{n-1}, z_1)$, and similarly for $y_{m-1}$. In particular,
  $d_X(x_{n-1}, y_{m-1}) \le 2(2R' + C + \delta)$, hence
  $d_X(g_{n-1},h_{m-1})\le 6R' + 2C + 4\delta$.

  A nearly identical argument applied to the points $x_n, y_m$ (with
  the roles of $\sigma_1, z_1$ interchanged with the roles of
  $\sigma_2, z_2$) shows that $\dcusp(g_n, h_m)$ is also at most
  $6R' + 2C + 4\delta$, meaning we can set $D_1 = 6R' + 2C + 4\delta$.
\end{proof}

\begin{proof}[Proof of Proposition \ref{prop:unif_nest}]
   Let $(g_0, \coding{e})$ and $(h_0, \coding{f})$ be generalized codings 
   of the same conical
  point $\zeta$, with $\lab(e_n) = \alpha_n$, $\lab(f_m) = \beta_m$
  and let $(g_n)_{n \in \bN \cup \{0\}}$ and $(h_m)_{m \in \bN \cup \{0\}}$ be the
  associated quasigeodesic sequences. 
  
  First, we set the constants $c$
  and $D_1$.  (The numbers $D_2,N$ depends on $\smallepsilon$ so they will be set later.)
  Choose
  $D_1$ and $J$ so that they are at least as large as the corresponding
  constants in Lemma \ref{lem:para_jump}, and so that $D_1$ is also at
  least the Hausdorff distance bound in
  Lemma~\ref{lem:codings_bounded_dist}.  Note that these only depend
  on $g_0$ and $h_0$, not $\coding{e}, \coding{f}$ or $\zeta$.
  
  By Lemma~\ref{lem:para_jump}, whenever $|\alpha_n|_X > J$, there
  exists a parabolic point $p \in \Pi$, an index $m$, and elements
  $g, h \in \Gamma$ such that $d(g_{n-1}, h_{m-1})< D_1$,
  $\head(e_n) = gp$ and $\head(f_m) = hp$, and the elements $g_{n-1}g$
  and $h_{m-1}h$ lie in the same coset of $\Gamma_p$.

  Thus, $g\Gamma_p = (g_{n-1}^{-1}h_{m-1}) h \Gamma_p$, and
  $|g_{n-1}^{-1}h_{m-1}|_X < D_1$, so $g_{n-1}^{-1}h_{m-1}$ lies in a
  finite set $F_1 := \{ f \in \Gamma \st |f|_X < D_1 \}$.  Restating the
  above, we have $gp = fhp$, for $f \in F_1$. Since $hp \in W(hp)$, we
  have $fhp = gp \in fW(hp)$.  Consider all possible triples $(z,y,f)$
  such that $z,y\in Z$, $f\in F_1$, and $z\in f W(y)$.
  There are finitely many such, so we may choose some $c>0$
  such that 
  $\overline{B}_{3c}(z) \subset f W(y)$ holds for every such triple.

  Now fix any $\smallepsilon< c$, and 
  choose $D_2' = D_2'(\smallepsilon) > J$ large enough so that, for each edge
  from $z$ to a vertex $x$ where $z = gp$ is parabolic, and each
  $\alpha \in L(z)$, if $|\alpha|_X > D_2'$ then
  $\alpha \overline{N}_{\bigepsilon}(W(x)) \subset B_{\smallepsilon}(z)$. We
  know such a $D_2'$ exists because only finitely many elements of the
  coset $g\Gamma_p$ fail this contraction condition: as we observed at
  the beginning of Remark~\ref{rem:finitely_many}, since the edge
  inclusion condition
  \[
    \alpha\overline{N}_\bigepsilon(W(x)) \subset \alpha\hat{W}(z) \subset
    W(z)
  \]
  holds for all but finitely many $\alpha \in g\Gamma_p$, the closed
  neighborhood $\overline{N}_\bigepsilon(W(x))$ cannot contain $p$ and
  therefore all but finitely many $\alpha \in g\Gamma_p$ take
  $\overline{N}_\bigepsilon(W(x))$ into an arbitrarily small neighborhood
  of $z$. We define
  $D_2 = D_2(\smallepsilon) := D_2'(\smallepsilon) + 2D_1$.  Again, this
  depends on $h_0$, $g_0$ and $\smallepsilon$, but not the strict
  codings $\coding{e}$, $\coding{f}$ or the point $\zeta$.
   
   %%%%%%%

  We now can prove that the generalized codings satisfy uniform nesting.  
  As a first case, suppose
  $|\alpha_n|_X > D_2'$ for infinitely many $n$. 
    We will show that,
  in this case, we have {\em uniform nesting with long parabolics}.
  Choose an infinite sequence of indices $n_k$ so that
  $|\alpha_{n_k + 1}|_X > D_2'$ is always satisfied. Lemma
  \ref{lem:para_jump} then provides a sequence of indices $m_k$ with
  $\dcusp(g_{n_k}, h_{m_k}) < D_1$ and
  $\dcusp(g_{n_k+1}, h_{m_k + 1}) < D_1$, which
  satisfies all of the requirements of ``uniform nesting with long
  parabolics'' by our choice of constants $D_2', \smallepsilon$. This
  finishes the proof in this case.
  
  On the other hand, if $|\beta_m|_X > D_2 = D_2' + 2D_1$ is satisfied
  for infinitely many $m$, then we can use Lemma~\ref{lem:para_jump}
  again to find an an infinite subsequence
  $(\alpha_{n_k})_{k \in \bN}$ so that $|\alpha_{n_k}|_X > D_2'$ for
  every $k$, and we are in the previous case.
  
  If neither of the first two cases hold, then we know that both
  $|\alpha_n|_X$ and $|\beta_m|_X$ are bounded by $D_2$ for all but
  finitely many $n, m$, respectively.  For this case, we will create
  new generalized codings of $\zeta$ by ``shifting the indices:'' we
  replace $g_0$ and $h_0$ with terms further along in the associated
  quasi-geodesic sequences, and truncating the first terms of the
  sequence so that all labels are bounded by $D_2$. This will put us
  in a position to consider a sub-automaton only including edges with
  short labels, and then apply Lemma \ref{lem:uniformnest} from the
  Appendix to conclude the proof.
  
  In more detail: first, choose $M$ large enough so that
  $|\alpha_n|_X \le D_2$ for all $n \geq M$.
  Lemma~\ref{lem:codings_bounded_dist} tells us that for each
  $n \in \bN$, there is some index $m(n)$ so that
  $d_X(g_n, h_{m(n)}) < D_1$.  Lemma~\ref{lem:boundedbacktrack_I}
  implies that $m(n)$ tends to infinity as $n$ tends to infinity, so
  increasing $M$ if necessary, we can ensure that
  $|\beta_{k}|_X \le D_2$ for every $k \geq m(M)$.

  Let $\coding{e}'$ be the sub-path of $\coding{e}$ starting with the
  edge $e_{M}$ and let $\coding{f}'$ be sub-path of $\coding{f}$
  starting with the edge $f_{m(M)}$. 
  Consider the generalized codings $(g_{M-1}, \coding{e}')$ and
  $(h_{m(M)-1}, \coding{f}')$, and let $g_n' = g_{n + M -1}$ and
  $h_m' = h_{m + m(M) -1}$ be their associated quasi-geodesic
  sequences. By construction, both of
  these codings are generalized codings of $\zeta$, and their
  associated quasi-geodesic sequences are tails of the associated
  quasi-geodesic sequences for $(g_0, \coding{e})$ and
  $(h_0, \coding{f})$, respectively. Further, the label sequences of
  $\coding{e}'$ and $\coding{f}'$ consist of elements whose length in
  $X$ is bounded by $D_2$.

  Let $\mc{F}$ be the subgraph of
  $\mc{G}$ obtained by deleting all labels of length more than $D_2$,
  and then deleting all edges with empty label set.  Then $\mc{F}$ is
  a \emph{finitary point coder} in the sense defined in
  Appendix~\ref{sec:appendix}, and $(g_{M-1}, \coding{e}')$ and
  $(h_{m(M)-1}, \coding{f}')$ are generalized $\mc{F}$-codings.

  Let $F$ be the finite subset of $\Gamma$ consisting of elements of
  length at most $D_1$, and apply this to Lemma~\ref{lem:uniformnest}.
  We conclude there is a constant $N > 0$, so that whenever
  $d_X(g_n', h_m') < D_1$, we have
  $g_{n + N}'\overline{W(z_{n+N})} \subset h_m'W(y_m)$, where
  $z_k = \tail(e'_k)$ and $y_k = \tail(f'_k)$.  Translating this
  statement, for each $m_k$ such that $d_X(g_{n_k}, h_{m_k}) \leq D_1$
  and $n_k > M$, we have
  \[g_{{n_k} + N}'\overline{W^{\coding{e}}(n_k +N)} \subset h_{m_k}
    W^{\coding{f}}(m_k). \]

  The constant $N$ in the above depends only
  on the point coder $\mc{F}$, which in turn depends only on $D_2$ and
  the original automaton $\mc{G}$, so this completes the proof.
\end{proof}

\begin{remark} \label{rem:finite_nesting_condition}  
  The containment 
  \begin{equation*}
    g_{n_k + N} \overline{W^\coding{e}(n_k+N)}\subset h_{m_k}W^\coding{f}(m_k)
  \end{equation*}
  from {\em uniform nesting with short words} of Definition \ref{def:unif_nest}
  is equivalent to 
  \begin{equation*}
    g_{n_k} \alpha_{n_k+1} \ldots \alpha_{n_k+ N} \overline{W^\coding{e}(n_k+N)}\subset h_{m_k}W^\coding{f}(m_k)
  \end{equation*}
  or, multiplying on the left by $g_{n_k}^{-1}$,
  \begin{equation}
    \label{eq:nest1}
    \alpha_{n_k+1} \ldots \alpha_{n_k+ N} \overline{W^\coding{e}(n_k+N)}\subset g_{n_k}^{-1}h_{m_k}W^\coding{f}(m_k) \tag{$\dagger$}
  \end{equation}
  Since the assumptions of uniform nesting with short words stipulate
  that $|\alpha_{n_k}| \le D_2$, and $|g_{n_k}^{-1}h_{m_k}|_X < R$,
  and since there are only finitely many sets of the form $W(z)$, the
  inclusions given by \eqref{eq:nest1} are only finite in number. So
  these inclusions correspond to finitely many open conditions,
  meaning they are stable under small perturbation. This is essential
  to our argument in Lemma \ref{lem:well-defined}.

  Similarly, the containments
  $g_{n_k} \overline{B}_{3\smallepsilon}(\head(e_{n_k+1}))\subset h_{m_k}
  W^\coding{e}(m_k+1)$ from the {\em uniform nesting with long
        parabolics} condition of Definition \ref{def:unif_nest} 
  are only
  finite in number. 
\end{remark} 

\section{Proof of main theorem} \label{sec:proof}
Thus far we have suppressed notation for the action of $\Gamma$ on its
boundary, simply writing $g(\zeta)$ or $g\zeta$ for the image of
$\zeta$ under $g$.  We will now need to consider other actions of
$\Gamma$ on this space, so we reintroduce the following notation.  

\begin{notation} 
As in the introduction, let
$\rho_0: \Gamma \to \Homeo(\bgamp)$ denote the standard action of $\Gamma$ on its Bowditch boundary.   Thus, what was previously written $g W(z_k)$ now becomes $\rho_0(g) W(z_k)$, for example.  
\end{notation}

Throughout this section we will work with a fixed automaton $\mc{G}$
for the $\rho_0$-action of $\Gamma$ on $\bgamp$, as given in
Definition \ref{def:automaton}. We also work with a fixed finite
generating set $\mc{S}$ for $\Gamma$.

\begin{definition}  \label{def:RV}
  Given a neighborhood of the identity $\mc{V} \subset C(\bgamp)$, we define
  $\semiconj_{\mc{V}} \subset \Hom(\Gamma, \Homeo(\bgamp))$ to be the
  set of representations $\rho$ such that, for each parabolic point
  $p \in \mathcal{P}$ with stabilizer $P = \Gamma_p$, there exists
  $\phi_{p} \in \mc{V}$ such that $\rho|_P$ has
  $\rho_0|_P$ as a topological factor, via $\phi_p$.
    \end{definition}

\begin{remark}
  \label{rem:parabolic_semiconj}
  The map $\phi_{p}$ defining a semi-conjugacy for the action of $P$ as in Definition \ref{def:RV} is a priori not determined by the
  representation $\rho \in \semiconj_{\mc{V}}$.  To get around this,
  for the rest of this section, whenever we fix
  $\rho \in \semiconj_{\mc{V}}$ for some $\mc{V} \subset C(\bgamp)$,
  then for each $p \in \Pi$ we implicitly choose a semi-conjugacy
  $\phi_{p} \in \mc{V}$ which extends the restriction of $\rho_0$ to
  $P = \Gamma_p$.
\end{remark}

\begin{notation} 
  Assuming that some action $\rho \in \semiconj_{\mc{V}}$ has been
  fixed, then for each parabolic point $z = \rho_0(g)p$, where
  $p \in \Pi$, we let $\phi_z := \rho_0(g) \phi_p \rho(g)^{-1}$.
\end{notation}

Observe the definition of $\phi_z$ depends only on $z$ and $\phi_p$,
and not on the choice of $g \in \Gamma$ such that $z = \rho_0(g)p$. 
To see this, suppose $\rho_0(g)p = \rho_0(h)p$.  Then 
\[\rho_0(g) \phi_p \rho(g)^{-1} = \rho_0(h)\rho_0(h^{-1}g) \phi_p \rho(h^{-1}g)^{-1} \rho(h^{-1}). \]
Since $h^{-1}g \in \Gamma_p$, 
$\rho_0(h^{-1}g) \phi_p = \phi_p \rho(h^{-1}g)$, which shows 
\[\rho_0(g) \phi_p \rho(g)^{-1} = \rho_0(h) \phi_p \rho(h)^{-1}\]
as desired.  

In addition, note that for fixed $z$, if $\rho$ approaches $\rho_0$,
and $\phi_p$ tends to the identity, then $\phi_z$ tends to the
identity as well. 
We also record the following.
\begin{lemma} \label{lem:coset_semiconjugacy} If $z = \rho_0(g)p$,
  then for any $x \in \bgamp$ and any $\alpha \in g\Gamma_p$, we have
\begin{equation}
  \label{eq:coset_semiconjugacy}
  \phi_z(\rho(\alpha)x) = \rho_0(\alpha)\phi_p(x).
\end{equation}
\end{lemma}
\begin{proof}
  Let $x$ be given, and suppose that $z = \rho_0(g)p$ and
  $\alpha = gh$ for $h \in \Gamma_p$. Then
\begin{align*}
  \phi_z(\rho(\alpha)x) &= \phi_z(\rho(gh)x)\\
                        &= \rho_0(g)\phi_p(\rho(g)^{-1}\rho(gh)x)\\
                        &= \rho_0(g)\phi_p(\rho(h)x)\\
                        &= \rho_0(gh)\phi_p(x) = \rho_0(\alpha)\phi_p(x).
\end{align*}
\end{proof} 

\subsection{Specifying the neighborhoods $\mc{U}$ and $\mc{V}'$}

\begin{definition}
  Given a subset $\mc{V} \subset C(\bgamp)$ and a representation
  $\rho \in \semiconj_{\mc{V}}$, we define families of sets
  $\{V_\rho(z)\}_{z \in Z}$ and $\{\hat{V}_\rho(z)\}_{z \in Z}$ as
  follows.
  \begin{itemize}
  \item If $z$ is a conical limit point, then we define
    \[
      V_\rho(z) = V(z), \qquad \hat{V}_\rho(z) =
      \rho(\alpha_z^{-1})V_\rho(z),
    \]
    where $\alpha_z$ is the unique element in the label set $L(z)$.
  \item If $z = \rho_0(g)p$ for some $g \in \Gamma$, $p \in \Pi$, we
    define
    \[
      V_\rho(z) = \phi_z^{-1}(V(z)),
    \]
    and define
    \[
      \hat{V}_\rho(z) = \phi_z^{-1}(\hat{V}(z)).
    \]
  \end{itemize}
\end{definition}

\begin{definition}
\label{defn:same_combinatorics}
  Suppose that $\mc{V} \subset C(\bgamp)$ and
  $\rho \in \semiconj_{\mc{V}}$. We say that $\rho$ \emph{has the same
    combinatorics} as the standard action $\rho_0$ if the following
  conditions hold:
  \begin{enumerate}[label=(\roman*)]
  \item\label{item:perturbed_covering} The collection $\{V_\rho(z)\}_{z \in Z}$ is an open
    covering of $\bgamp$.
  \item\label{item:perturbed_v_in_w} For every $z \in Z$, we have $V_\rho(z) \subset W(z)$.
  \item\label{item:perturbed_intersections} For any $y, z \in Z$, we have $\hat{V}_\rho(y)
    \cap V_\rho(z) = \emptyset$ if and only if $\hat{V}(y) \cap V(z) =
    \emptyset$.
  \item\label{item:perturbed_edge_nesting} If there is an edge from
    $z$ to $y$ in $\mc{G}$ labeled by $\alpha$, then
    $ \rho(\alpha)(W(y)) \subset W(z).$
  \end{enumerate}
\end{definition}

\begin{lemma}[Same combinatorics is relatively open]
  \label{lem:same_comb_open}
  For any sufficiently small neighborhood $\mc{V'}$ of the identity in
  $C(\bgamp)$ and any sufficiently small neighborhood $\mc{U}$ of $\rho_0$ in
  $\Homeo(\Gamma, \bgamp)$, each  $\rho \in \mc{U} \cap \semiconj_{\mc{V}'}$ has the same
  combinatorics as $\rho_0$.
\end{lemma} 
  
\begin{proof}
  First, we ensure that
  Item~\ref{item:perturbed_covering} in the definition holds. Let $r$
  be a Lebesgue number for the original open covering
  $\{V(z)\}_{z \in Z}$, so that every set of diameter at most $r$ is
  contained in some $V(z)$. We may choose $\mc{V}'$ small enough so
  that $\dvis(\phi_z(x), x) < r/2$ for every parabolic point $z \in Z$
  and $x \in \bgamp$.

  Then, for any $x \in \bgamp$, the $(r/2)$-ball about $x$ is
  contained in some $V(z)$, where $z \in \bgamp$ is either conical or
  parabolic. If $z$ is conical then $V_\rho(z) = V(z)$ and thus
  $x \in V_\rho(z)$. Otherwise, if $z$ is parabolic, then $\phi_z(x)$
  lies in $B_{r/2}(x) \subset V(z)$, hence
  $x \in \phi_z^{-1}(V(z)) = V_\rho(z)$.

  Item~\ref{item:perturbed_v_in_w} can be arranged because
  $\overline{V(z)}\subset W(z)$ for each $z$ (see
  Condition~\ref{item:edgedefinition}).  In particular, there is some
  minimum distance from any $V(z)$ to $\bgamp\minus W(z)$, and we can
  choose $\mc{V'}$ small enough so that no $\phi_z$ moves any point
  more than that distance.

  For Item~\ref{item:perturbed_intersections}, we argue
  similarly. Property~\ref{item:stable_intersections} from
  Proposition~\ref{prop:C_properties} implies there is a minimum
  distance from any $V(z)$ to any disjoint $\hat{V}(y)$.  Shrinking
  $\mc{V'}$ if necessary, we may assume that no $\phi_z$ moves any
  point more than half that distance.  By choosing sufficiently small
  $\mc{U}$, we can ensure that the distance between $\rho(\alpha_z)$
  and $\rho_0(\alpha_z)$ is also at most half that distance.  This
  will ensure that $V(z)\cap \hat{V}(y) = \emptyset$ implies
  $V_\rho(z)\cap\hat{V}_\rho(y) = \emptyset$. For the converse,
  observe that there is some positive radius $r$ so that each nonempty
  intersection $V(z) \cap \hat{V}(y)$ contains a ball of radius
  $r$. The sets $V_\rho(z)$ and $\hat{V}_\rho(y)$ each contain the
  preimage of this ball by a continuous map arbitrarily close to the
  identity, so if $\mc{V}'$ and $\mc{U}$ are small enough then these
  preimages have nonempty intersection.
  
  We now turn to condition \ref{item:perturbed_edge_nesting}. By
  condition \eqref{eq:edge_inclusions} in
  Proposition~\ref{prop:proper_nesting} we have
  $\rho_0(\alpha) \overline{N}_\bigepsilon(W(y)) \subsetneq W(z)$ for
  every edge from $z$ to $y$ labeled by $\alpha$. For conical vertices
  $z$, provided $\rho$ is a small enough perturbation of $\rho_0$,
  each of the containments $\rho(\alpha)(W(y)) \subset W(z)$ will
  hold, since there are only finitely many such edges and labels. For
  parabolic vertices, we argue as follows.

  Consider the constants $\epsilon_z$ from Remark
  \ref{rem:finitely_many}, and fix
  $\epsilon_{\min} \leq \min\{\epsilon_z \st z \in Z \text{
    parabolic}\}$.  The inclusion \eqref{eq:ez} from Remark
  \ref{rem:finitely_many}
    says if $z$ is connected to $y$ by an edge
  labeled $\alpha$, we have
   \begin{equation} \label{bigger_nbhd}
     N_{\epsilon_{\min}}(\rho_0(\alpha)(N_{\bigepsilon} W(y))) \subset
     W(z).
   \end{equation}
   We choose our neighborhood $\mathcal{V'}$ small enough so that, for
   each of the finitely many parabolic vertices $z \in Z$, the
   semi-conjugacy $\phi_z$ satisfies
   $\dvis(x, \phi_z(x)) < \min\{\epsilon_{\min}, \bigepsilon\}$. This
   ensures that for any subset $A \subset \bgamp$ we have
   $\phi_z(A) \subset N_{\bigepsilon}(A)$ and
   $\phi^{-1}_z(A) \subset N_{\epsilon_{\min}}(A)$.  In particular, we
   have $\phi_z W(y) \subset N_{\bigepsilon}( W(y))$, so our choice of
   $\epsilon_{\min}$ and the containment from \eqref{bigger_nbhd}
   above ensures that
   \[ 
     \rho(\alpha) W(y) \subset \phi_z^{-1} \phi_z \rho(\alpha) W(y)  \subset N_{\epsilon_{\min}}(\rho_0(\alpha \phi_z W(y))
     \subset W(z) 
   \] 
   as desired.
\end{proof}

Our next step is to define \emph{$(\mc{G}, \rho)$-codings}, which
provide a modified notion of a $\mc{G}$-coding which is compatible
with the perturbed action $\rho$ instead of the standard action
$\rho_0$. A $(\mc{G}, \rho_0)$-coding is the same thing as a
$\mc{G}$-coding, but since our convention in this section is to make
the standard action explicit, we will only refer to these as
$(\mc{G}, \rho_0)$-codings from this point forward. 

\begin{definition}
  Let $\mc{V} \subset C(\bgamp)$, and suppose that
  $\rho \in \semiconj_{\mc{V}}$ has the same combinatorics as
  $\rho_0$. If $\coding{e} = (e_k)_{k \in \bN}$ is an infinite
  edge path in $\mc{G}$ with $\lab(e_k) = \alpha_k$, we say
  $\coding{e}$ is a $(\mc{G}, \rho)$-coding for $\zeta$ if
  \[
    \zeta \in \bigcap_{k=1}^\infty \rho(\alpha_1) \cdots
    \rho(\alpha_k) W(z_k).
  \]

  If $e_1 \ldots e_n$ is a finite edge path giving a $(\mc{G}, \rho_0)$-coding
  for a parabolic point $z \in \bgamp$, we say that it is a
  $(\mc{G}, \rho)$-coding for $\zeta$ if $\zeta \in \phi_z^{-1}(z)$,
  or equivalently if
  $\zeta \in \rho(\alpha_1 \cdots \alpha_n)\phi_{z_n}^{-1}(z_n)$,
  where $z_n = \tail(e_n)$.
\end{definition}

\begin{lemma}\label{lem:perturbedcodingsexist}
  Let $\mc{V} \subset C(\bgamp)$, and suppose that
  $\rho \in \semiconj_{\mc{V}}$ has the same combinatorics as
  $\rho_0$. Then every point in $\bgamp$ has a
  $(\mc{G}, \rho)$-coding.
\end{lemma}

\begin{proof}
  The proof follows the same strategy of Lemma
  \ref{lem:codings_exist}.  The first part of the inductive procedure
  goes through verbatim, replacing the sets $V(z)$ with $V_\rho(z)$,
  and elements $\alpha_i$ with $\rho(\alpha_i)$.  The only
  modification required occurs at the inductive step when
  $z_{n} = \tail(e_n)$ is a parabolic rather than conical limit
  point.  In this case we simply need to pay attention to the
  semi-conjugacies $\phi_{z_n}$.

  In detail, adopting the
  notation and setting from the proof of Lemma
  \ref{lem:codings_exist}, suppose we are given $\zeta \in \bgamp$,
  and assume we have found a partial coding so that
  $\zeta \in \rho(\alpha_1 \cdots \alpha_n) V_\rho(z_n)$ where $z_n$
  is a parabolic point. If
  $\zeta \in \rho(\alpha_1 \cdots \alpha_n)\phi_{z_n}^{-1}(z_n)$, then
  we have found a parabolic $(\mc{G}, \rho)$-coding for $\zeta$ and
  are done.  Otherwise,
  $\zeta \notin \rho(\alpha_1 \cdots \alpha_n)\phi_{z_n}^{-1}(z_{n})$.
  Let $\zeta_n := \rho(\alpha_1 \cdots \alpha_n)^{-1}\zeta$, so that
  $\zeta_n \in V_\rho(z_n) \minus \phi_{z_n}^{-1}(z_n)$. Then
  $\phi_{z_n}(\zeta_n) \in V(z_n) \minus \{z_n\}$ by the definition of
  $V_\rho(z_n)$. So, by property \ref{item:edgedefinition} of the
  $\rho_0$-automaton, there is some $\alpha_{n+1} \in L(z_n)$ so that
  $\rho_0(\alpha_{n+1}^{-1})\phi_{z_n}(\zeta_n) \in \hat{V}(z_n)$.

  We can then apply Lemma \ref{lem:coset_semiconjugacy} (with
  $x = \rho(\alpha_{n+1}^{-1})\zeta_n$) to see that
  \[
    \rho_0(\alpha_{n+1}^{-1})\phi_{z_n}(\zeta_n) =
    \phi_p(\rho(\alpha_{n+1}^{-1})\zeta_n).
  \]
  We conclude that
  $\phi_p(\rho(\alpha_{n+1}^{-1}\zeta_n)) \in \hat{V}(z_n)$. So,
  setting $\zeta_{n+1} := \rho(\alpha_{n+1}^{-1})\zeta_n$, we have
  \[
    \zeta_{n+1} \in \phi_p^{-1}(\hat{V}(z_n)) = \hat{V}_\rho(z_n).
  \]
  Since the $V_\rho(z)$ sets still cover $\bgamp$, there is some
  $z_{n+1}$ with $\zeta_{n+1} \in V_\rho(z_{n+1})$. Since $\rho$ has
  the same combinatorics as $\rho_0$, and $\zeta_n$ lies in
  $\hat{V}_\rho(z_n) \cap V_\rho(z_{n+1})$ there is an edge in
  $\mc{G}$ from $z_n$ to $z_{n+1}$, which completes the inductive
  step.
\end{proof}

\begin{convention}[Choosing constants] \label{conv:NDeps}
For each pair $s, s'$ in $\mathcal{S} \cup \{\idgamma\}$, Proposition \ref{prop:unif_nest} gives a constant $c(s, s')$ such that any generalized codings of a common point of the form $(s, \coding{e})$ and $(s', \coding{f})$ satisfy {\em $c$-uniform nesting}.  Fix $\smallepsilon < \min\{ c(s, s') \st s, s' \in \mathcal{S} \cup \{\idgamma\} \}$. 

Proposition \ref{prop:unif_nest} now states that, for each pair
$(s, s')$, for this fixed $\smallepsilon$ there exist constants
$N(s,s'), D_1(s,s')$ and $D_2(s,s')$ so that, for every pair of
generalized codings $(s, \coding{e})$ and $(s', \coding{f})$ of a
common conical point, uniform nesting is satisfied with the constants
$N(s,s'), D_1(s,s')$ and $D_2(s,s')$ and $\smallepsilon$.  Fix $N$, $D_1$
and $D_2$ greater than the maximum of all such respective constants
ranging over all pairs $s,s' \in \mc{S} \cup \{\idgamma\}$.
\end{convention} 

We note that these constants also implicitly depended on $\bigepsilon$, which was specified in Definition \ref{def:epsilon} using our target neighborhood $\mc V$, and will reappear in Lemma \ref{lem:close_to_id}.

\begin{lemma}
  \label{lem:open_nbhds_exist}
  There exists an open neighborhood $\mc{U}$ of $\rho_0$ in
  $\Hom(\Gamma, \bgamp)$ and an open neighborhood $\mc{V}'$ of the
  identity in $C(\bgamp)$ such that for any
  $\rho \in \mc{U} \cap \semiconj_{\mc{V}'}$, the following hold.   
  \begin{enumerate}[label=(V\arabic*)]
  \item \label{item:same_combinatorics} The representation $\rho$ has
    the same combinatorics as $\rho_0$.
  \item \label{item:short_word_stable_nest} 
  For any 
    $g \in \Gamma$
    satisfying $|g|_X \le D_1 + D_2 N$, 
    and any $y, z \in Z$, if
    $\rho_0(g) \overline{W(z)} \subset  W(y)$, then
    $\rho(g) \overline{W(z)} \subset W(y)$.
      \item \label{item:parabolic_nbhd_stable_nest} 
      For any 
      $q, y \in Z$ with $q$ parabolic, and any $g \in \Gamma$
    satisfying $|g|_X \le D_1$, if
    $\rho_0(g)\overline{B}_{3\smallepsilon}(q) \subset  
    W(y)$, then
    $\rho(g)\overline{B}_{3\smallepsilon}(q) \subset 
    W(y)$.
  \item \label{item:long_parabolic_stable_nest} For every parabolic
    vertex $q \in Z$ and every edge in $\mc{G}$ from $q$ to $y$
    labeled by $\alpha$, if
    $\rho_0(\alpha) N_{\bigepsilon}(W(y)) \subset B_{\smallepsilon}(q)$,
    then $\rho(\alpha)W(y) \subset B_{3\smallepsilon}(q)$.
  \end{enumerate}
\end{lemma}
\begin{proof}
  Lemma~\ref{lem:same_comb_open} shows that having 
  the same combinatorics is a relatively open condition, thus we
  can find $\mc{U}$, $\mc{V}'$ so that Item
  \ref{item:same_combinatorics} holds. Both
  \ref{item:short_word_stable_nest} and
  \ref{item:parabolic_nbhd_stable_nest} correspond to open conditions
  on finitely many elements in $\Gamma$.
  Thus, we need only demonstrate that by further shrinking 
  $\mc{V}'$ and $\mc{U}$ if needed, we may satisfy
  \ref{item:long_parabolic_stable_nest}.  

  As in the proof of Lemma \ref{lem:para_jump}, for each parabolic vertex $q \in Z$, we choose some $t_q \in \Gamma$ 
  so that $q = \rho_0(t_q)p$ for some $p \in \Pi$. We choose our
  neighborhood $\mc{U}$ of $\rho_0$ small enough so that for any $\rho \in \mc{U}$
  and every parabolic vertex $q \in Z$, we have
  \begin{equation}
    \label{eq:coset_inclusion}
    \rho(t_q)  \rho_0(t_q)^{-1}B_{2\smallepsilon}(q) \subset
    B_{3\smallepsilon}(q).
  \end{equation}

  Choose $\mc{V}'$ sufficiently small so that for any
  $\phi \in \mc{V}'$ and every $y \in Z$, we have 
 \begin{equation} \label{eq:chooseVprime}
 \phi(W(y)) \subset N_{\bigepsilon}(W(y))
 \end{equation} 
   and for every parabolic vertex 
  $q = \rho_0(t_q)p \in Z$, we have 
  \[ \phi^{-1}(\rho_0(t_q))^{-1}B_{\smallepsilon}(q)) \subset \rho_0(t_q)^{-1}B_{2\smallepsilon}(q).\]
  
  Let $\rho \in \mc{U} \cap \semiconj_{\mc{V}'}$. Fix a parabolic
  vertex
  $q = t_q p \in Z$. Suppose that for some edge $e$ from $q$ to
  $y$ labeled by $\alpha \in L(q) \subset t_q\Gamma_p$, we have
  $\rho_0(\alpha) N_{\bigepsilon}(W(y)) \subset B_{\smallepsilon}(q)$. We
  may write $\alpha = \rho_0(t_q)\alpha'$\ 
  for some $\alpha' \in \Gamma_p$, so
  $\rho_0(\alpha')N_{\bigepsilon}(W(y)) \subset
  \rho_0(t_q)^{-1}B_{\smallepsilon}(q)$.
  
  Also by our choice of $\mc{V}'$ and \eqref{eq:chooseVprime}, we have
$\phi_p(W(y)) \subset N_{\bigepsilon}(W(y)) $ for all $p \in \Pi$ and $y \in Z$.   

  From this it follows that 
  \begin{align*}
    \rho(\alpha')W(y) &\subset
                        \rho(\alpha')\phi_p^{-1}N_{\bigepsilon}(W(y))\\
                      &= \phi_p^{-1}(\rho_0(\alpha') N_{\bigepsilon}(W(y)))\\
                      &\subset \phi_p^{-1}(\rho_0(t_q)^{-1}B_{\smallepsilon}(q))\\
                      &\subset \rho_0(t_q)^{-1}B_{2\smallepsilon}(q).
  \end{align*}
  Putting this together with \eqref{eq:coset_inclusion}, we conclude that 
  \[
    \rho(\alpha)W(y) = \rho(t_q)\rho(\alpha')W(y) \subset \rho(t_q)
    \rho_0(t_q)^{-1}B_{2\smallepsilon}(q) \subset B_{3\smallepsilon}(q)
  \]
  as desired.  
\end{proof}

\subsection{Defining the semi-conjugacy}

For the rest of the paper, we fix open sets
$\mc{U} \subset \Hom(\Gamma, \Homeo(\bgamp))$ and
$\mc{V}' \subset C(\bgamp)$ satisfying the conditions of Lemma
\ref{lem:open_nbhds_exist}. We let $\rho:\Gamma \to \Homeo(\bgamp)$ be
a representation in $\mc{U} \cap \semiconj_{\mc{V}'}$. We recall from
Remark~\ref{rem:parabolic_semiconj} that this means that for each
$p \in \Pi$, we fix a semi-conjugacy $\phi_{p} \in \mc{V}'$,
such that $\rho|_{\Gamma_p}$ has $\rho_0|_{\Gamma_p}$ as a factor, via
$\phi_{p}$. In turn this determines semi-conjugacies $\phi_z$
extending $\rho|_{\Gamma_z}$ for each parabolic $z \in \bgamp$ with
stabilizer $\Gamma_z$.

Our goal is to show that $\rho$ is semi-conjugate to $\rho_0$ via some
map $\phi$. We now set about constructing a map $\phi$, and will then
show that it is indeed a well-defined semi-conjugacy, and in the
neighborhood $\mc{V}$ of the identity that was fixed in the first paragraph of Section
\ref{subsec:fixing}.  The general strategy is to use codings to assign
a well-defined closed subset $\Phi(\zeta) \subset \bgamp$ to each
$\zeta \in \bgamp$, then specify that $\phi$ send $\Phi(\zeta)$ to
$\zeta$.  To ensure this gives a well-defined continuous map, we need
the following lemma.

\begin{lemma}
  \label{lem:well-defined}
  Let $\zeta$ be a conical limit point in $\bgamp$, and
  $s \in \mathcal{S} \cup \{\idgamma\}$. Suppose that
  $\coding{e} = (e_k)_{k \in \bN}$ is a strict $(\mc{G}, \rho_0)$-coding
  of $\rho_0(s)\zeta$, and $\coding{f} = (f_k)_{k \in \bN}$ is a strict
  $(\mc{G}, \rho_0)$-coding of $\zeta$; equivalently $(s, \coding{f})$
  is a generalized $(\mc{G}, \rho_0)$-coding of $\rho_0(s)\zeta$.  Then, for $\alpha_k = \lab(e_k)$ and
  $\beta_k = \lab(f_k)$, we have 
  \[
    \bigcap_{k=1}^\infty \rho(\alpha_1) \cdots \rho(\alpha_k)
    \overline{W(\tail(e_{k}))} = \rho(s)\bigcap_{k=1}^\infty \rho(\beta_1)
    \cdots \rho(\beta_k) \overline{W(\tail(f_{k}))}.
  \]
\end{lemma}
\begin{proof}
  For each $k \in \bN$, write $W^{\coding{e}}(k) = W(\tail(e_k))$ and
  $W^{\coding{f}}(k) = W(\tail(f_k))$.  Let $g_k$ and $h_k$ be the
  associated quasi-geodesic sequences to the codings
  $\coding{e} = (\idgamma, \coding{e})$ and $(s, \coding{f})$,
  respectively (so $h_0 = s$).  Then
  $h'_k := s^{-1}h_k = \beta_1 \cdots \beta_k$ is the associated
  quasi-geodesic sequence to the strict coding
  $\coding{f}$.

  We will prove the inclusion
  \begin{equation}
    \label{eq:perturbed_nesting_2}
    \bigcap_{k=1}^\infty \rho(g_k)\overline{W^{\coding{e}}(k )}
    \subseteq \rho(s)\bigcap_{k=1}^\infty \rho(h'_k)\overline{W^{\coding{f}}(k)}.
  \end{equation}
  Since $\mathcal{S}$ is symmetric, the other inclusion then follows
  immediately. 
  
  Since both intersections in
  \eqref{eq:perturbed_nesting_2} are given by decreasing sequences of
  sets, it suffices to show that we can find arbitrarily large pairs of
  indices $L, R$ so that
  \[
    \rho(g_L)\overline{W^{\coding{e}}(L)} \subseteq
    \rho(s)\rho(h'_R)\overline{W^{\coding{f}}(R)}. 
  \]

  Since both $(\idgamma, \coding{e})$ and $(s, \coding{f})$ are
  generalized $(\mc{G}, \rho_0)$-codings of the point $\rho_0(s)\zeta$,
  Proposition \ref{prop:unif_nest} and our choice of constants imply
  that $(\idgamma, \coding{e})$ and $(s, \coding{f})$ have the uniform
  nesting property, with respect to the constants chosen in Convention
  \ref{conv:NDeps}.  Thus, there exist sequences $n_k$, $m_k$
  satisfying the properties in Definition \ref{def:unif_nest} for the
  corresponding pair of quasi-geodesic sequences.  Let $L$ be one such
  choice of $n_k$, and let $R$ denote $m_k$ (for the same $k$).  Thus,
  $|h_R^{-1}g_L|_X \leq D_1$.

If {\em uniform nesting with short words} holds, we have:
  \begin{equation}
    \rho_0(g_{L + N})\overline{W^{\coding{e}}(L + N )} \subset
    \rho_0(h_R) W^{\coding{f}}(R) 
  \end{equation}
  and $|\alpha_{L + k}|_X \le D_2$ for all $k \in \mathbb{N}$. 
  
%%%%
  Let $g_{(L + 1, L + N)} = \alpha_{L + 1} \cdots
  \alpha_{L + N}$, so that
  \[
    g_{L + N} = g_L \cdot  g_{(L + 1, L + N)}.
  \]  
  
  Thus,
  \[
    \rho_0((h_R)^{-1}g_L) \rho_0(g_{(L + 1, L + N)})
    \overline{W^{\coding{e}}(L + N)} \subset W^{\coding{f}}(R
    ).
  \]
    Since each $\alpha_{L + k}$ has length at most $D_2$, we know
  $|g_{(L + 1, L + N)}|_X \le ND_2$.  By definition, $|h_R^{-1}g_L|_X \leq D_1$, and therefore
  $|(h_R)^{-1}g_L\, g_{(L + 1, L + N)}|_X \le D_1 + ND_2$
  
Thus, by condition  \ref{item:short_word_stable_nest} of Lemma \ref{lem:open_nbhds_exist} and our choice of $\mathcal{U}$, we have 
    \[ \rho(g_{L + N})\overline{W^{\coding{e}}(L + N )} \subset
    \rho(h_R) W^{\coding{f}}(R) = \rho(s) \rho( h'_R) W^{\coding{f}}(R) \]
 which is what we needed to show.

  If instead we have uniform nesting with long parabolics, then
  $z_L = \tail(e_L) = \head(e_{L+1})$ is parabolic, and we have  
  \begin{equation}
    \label{eq:translate_par_nbhd}
    \rho_0(g_L)\overline{B}_{3\smallepsilon}(z_L) \subset \rho_0(h_R)W^{\coding{f}}(R),
  \end{equation}
  \begin{equation}
    \label{eq:deep_nest_before_perturb}
    \rho_0(\alpha_{L + 1})N_{\bigepsilon}(W^{\coding{e}}(L + 1)) \subset B_{\smallepsilon}(z_L).
  \end{equation}
  Since $|(h_R)^{-1}g_L|_X \le D_1$, \eqref{eq:translate_par_nbhd},
  together with condition \ref{item:parabolic_nbhd_stable_nest} on the
  neighborhood $\mc{U}$, implies that
  \[
    \rho(g_L)\overline{B}_{3\smallepsilon}(z_L) \subset
    \rho(h_R)W^{\coding{f}}(R).
  \]
  And, our assumption \ref{item:long_parabolic_stable_nest}, together
  with \eqref{eq:deep_nest_before_perturb}, implies that
  \[
    \rho(\alpha_{L + 1})W^{\coding{e}}(L + 1) \subset
    B_{3\smallepsilon}(z_{L}).
  \]
  We can then conclude that
  \[
    \rho(g_{L + 1})W^{\coding{e}}(L + 1) \subset
    \rho(h_R)W^{\coding{f}}(R).
  \]
  This means that
  $\rho(g_{L + 2})\overline{W^{\coding{e}}(L + 2)} \subset
  \rho(h_R)W^{\coding{f}}(R) = \rho(s)\rho(h'_R)W^{\coding{f}}(R)$, and we are done.  
\end{proof}

\begin{definition}
  Define a map $\Phi$ from $\bgamp$ to the space of closed subsets
  of $\bgamp$ as follows:
  \begin{itemize}
  \item If $p \in \bgamp$ is a conical limit point, choose a strict
    ($\mc{G}$, $\rho_0$)-coding of $p$ with terminal vertex sequence 
    $(z_k)_{k \in \bN}$ and label sequence $(\alpha_k)_{k \in \bN}$,
    and define
    \[
      \Phi(p) = \bigcap_{k=1}^\infty \rho(\alpha_1 \cdots
      \alpha_k)\overline{W(z_k)}.
    \]
  \item If $q \in \bgamp$ is a parabolic point, then we choose some
    $g \in \Gamma$ so that $q = \rho_0(g)p$ for a point $p \in
    \Pi$. Then define
    \[
      \Phi(q) = \rho(g)\phi_p^{-1}(p).
    \]
  \end{itemize}
\end{definition}

Taking $s= \idgamma$ in Lemma \ref{lem:well-defined} shows that the map $\Phi$ is
well-defined on conical limit points.  This lemma also implies that
$\Phi(\rho_0(s)z) = \rho(s)\Phi(z)$ for any $s \in \mathcal{S}$ and
conical $z$, which means that $\Phi$ is equivariant on conical limit
points.

To show $\Phi$ is well-defined on parabolic points, suppose
a parabolic point $q$ satisfies $q = \rho_0(h)p_i = \rho_0(g)p_j$ for
$p_i, p_j \in \Pi$.  Let $P_i = \Gamma_{p_i}$ and  $P_j  = \Gamma_{p_j}$.
Then $p_i = p_j$ and $h \in gP_i$, so
$\rho(g)\phi_{P_i}^{-1}(p_i) = \rho(h)\phi_{P_j}^{-1}(p_j)$. The 
same reasoning shows that in fact $\Phi$ is equivariant on parabolic
points: for any $g \in \Gamma$ and parabolic $q \in \bgamp$, we have
$\rho(g)\Phi(q) = \Phi(\rho_0(g) q)$.
We record this fact for future use.  

\begin{lemma}
  \label{lem:equivariant}
  The map $\Phi$ is equivariant, in the sense that for any
  $g \in \Gamma$ and $\zeta \in \bgamp$, we have
  \[
    \Phi(\rho_0(g) \zeta) = \rho(g)\Phi(\zeta).
  \]
\end{lemma}

Our next goal is to show that the sets $\Phi(\zeta)$ partition $\bgamp$.  
Towards this, we first observe that the sets $W(z)$ for $z \in Z$ can be
used to ``approximate'' the map $\Phi$. 

\begin{lemma}
  \label{lem:coding_approximates_point}
  Let $\zeta \in \bgamp$, and let $\coding{e}$ be a strict ($\mc{G}$,
  $\rho_0$)-coding for $\zeta$ whose first vertex is $z_0 \in Z$. Then
  we have both $\zeta \in W(z_0)$ and
  $\Phi(\zeta) \subset W(z_0)$.
\end{lemma}
\begin{proof}
  If $\zeta$ is a conical point, the first conclusion is immediate
  from the definition of a strict $(\mc{G}, \rho_0)$-coding, and the
  second follows from the definition of $\Phi$ and the fact that
  $\rho$ has the same combinatorics as $\rho_0$. So suppose $\zeta$ is
  parabolic, and consider a finite $(\mc{G}, \rho_0)$-coding
  $\coding{e}$ for the parabolic point $\zeta$ (which exists by
  Corollary \ref{cor:para_coding}), with label sequence
  $\alpha_1, \ldots, \alpha_n$ and vertex sequence $z_0, \ldots,
  z_n$. Then $z_n$ is parabolic and
  $\zeta = \rho_0(\alpha_1 \cdots \alpha_n)z_n$. By property
  \ref{item:edgenest} of our original automaton, we have
  $z_n \in W(z_n)$. Applying \eqref{eq:edge_inclusions} inductively we
  get
  $\zeta \in \rho_0(\alpha_1 \cdots \alpha_n)W(z_n) \subset W(z_0)$,
  proving that the first conclusion holds.

  Now we turn to the second conclusion. By equivariance, we have
  $\Phi(\zeta) = \rho(\alpha_1 \cdots \alpha_n)\Phi(z_n)$. We know
  $z_n \in V(z_n)$ by Property \ref{item:edgedefinition} of our
  original automaton, so
  $\Phi(z_n) = \phi_{z_n}^{-1}(z_n) \subset V_\rho(z_n)$. Then,
  because $\rho$ has the same combinatorics as $\rho_0$, we know
  $\Phi(z_n) \subset W(z_n)$ (from part \ref{item:perturbed_v_in_w} of
  the definition) and therefore $\Phi(\zeta) \subset W(z_0)$ (from
  inductively applying part \ref{item:perturbed_edge_nesting}).
\end{proof}

The endgame of the proof of Theorem \ref{thm:main} is identical to that 
in the case without parabolics, given at the end of Section 4 in
\cite{MMW1}. For convenience, we repeat it below.

\begin{lemma}
  \label{lem:inverse_disjoint}
  For any distinct $a, b \in \bgamp$, the sets $\Phi(a), \Phi(b)$ are
  disjoint.
\end{lemma}
\begin{proof}
  Given $a \neq b \in \bgamp$, let $g \in \Gamma$ be such that
  $\dvis(\rho_0(g)a, \rho_0(g)b) > D$, where $D$ is the constant from
  Definition \ref{def:pair_separation_constant}.  Let $z_0$ and $y_0$
  be the initial vertices of strict $\mc{G}$-codings for $\rho_0(g)a$
  and $\rho_0(g)b$, respectively. From
  Lemma~\ref{lem:coding_approximates_point}, we have
  $\rho_0(g)a \in W(z_0)$ and $\rho_0(g)b \in W(y_0)$.  Since each set
  $W(z)$ has diameter strictly less than $D/4$, and
  $\dvis(\rho_0(g)a, \rho_0(g)b) > D$, we conclude that
  $W(z_0) \cap W(y_0) = \emptyset$.
  
  Applying Lemmas~\ref{lem:equivariant} and
  \ref{lem:coding_approximates_point}, we have
\[\Phi(\rho_0(g)a) = \rho(g)\Phi(a) \subset W(z_0)\] 
and similarly
  $\rho(g)\Phi(b) \subset W(y_0)$.  Since 
  $W(z_0) \cap W(y_0) = \emptyset$, we also have
  $\Phi(a) \cap \Phi(b) = \emptyset$.
\end{proof} 

The fact that every point in $\bgamp$ has a strict
$(\mc{G}, \rho)$-coding (Lemma~\ref{lem:perturbedcodingsexist})
ensures that the union of the sets $\Phi(z)$ for $z \in \bgamp$ is all
of $\bgamp$
and Lemma~\ref{lem:inverse_disjoint}
ensures that this union is in fact a partition. So, the following
definition makes sense and gives a surjective map.  

\begin{definition}
  Define $\phi:\bgamp \to \bgamp$ by taking
  $\phi(x) = y$ if $x \in \Phi(y)$.
\end{definition}

We know that $\phi$ is equivariant by Lemma~\ref{lem:equivariant}, so
to prove that it is a semi-conjugacy lying in our chosen neighborhood
$\mc{V} \subset C(\bgamp)$, we just need to prove that it is a
continuous map which is $\bigepsilon$-close to the identity (see Definition \ref{def:epsilon}).

\begin{lemma} \label{lem:close_to_id}
  For every $\zeta \in \bgamp$, we have
  $\dvis(\zeta, \phi(\zeta)) < \bigepsilon$.
\end{lemma}
\begin{proof}
  Fix $\zeta \in \bgamp$ and let $\xi = \phi(\zeta)$, so
  $\zeta \in \Phi(\xi)$. Then by
  Lemma~\ref{lem:coding_approximates_point}, if $\coding{e}$ is a
  $(\mc{G}, \rho_0)$-coding for $\xi$ with initial vertex $z_0$, we
  have both $\xi \in W(z_0)$ and $\Phi(\xi) \subset W(z_0)$, hence
  $\zeta \in W(z_0)$. Since $\diam(W(z_0)) < \bigepsilon$ by property
  \ref{item:smalldiam}, the result follows.
\end{proof}

\begin{lemma} \label{lem:continuity} 
  The map $\phi$ is continuous.
\end{lemma}
\begin{proof}
  Fix $\zeta \in \bgamp$. By the previous lemma we know that
  $\dvis(\zeta, \phi(\zeta)) < \bigepsilon$. To show continuity at
  $\zeta$ we will use equivariance and the convergence action of
  $\Gamma$ on $\bgamp$.  Suppose $\zeta_n \to \zeta$, but along some
  subsequence we have $\phi(\zeta_n) \to \xi \neq \phi(\zeta)$.  By
  equivariance of $\phi$, we may assume without loss of generality
  (after applying some $g \in \Gamma$) that
  $\dvis(\xi, \phi(\zeta)) > D$, where $D$ is the constant from Definition \ref{def:pair_separation_constant}.  On the other hand, by 
  Lemma \ref{lem:close_to_id} and the
  triangle inequality, we have 
    \[
    \dvis(\phi(\zeta_n), \phi(\zeta)) \le \dvis(\phi(\zeta_n), \zeta_n)
    + \dvis(\zeta_n, \zeta) + \dvis(\zeta, \phi(\zeta)) < 3\bigepsilon
  \] 
  provided that $n$ is sufficiently large. Then
  $\dvis(\xi, \phi(\zeta_n)) > D - 3\bigepsilon > \bigepsilon$ (recall
  from Definition~\ref{def:epsilon} that $\bigepsilon \le D/5$). This
  gives a contradiction.
\end{proof} 

This completes the demonstration that $\phi$ is a semi-conjugacy
satisfying the properties of Theorem \ref{thm:main}, and concludes the
proof.

\appendix

\section{Uniform nesting for automata}
\label{sec:appendix}

The purpose of this appendix is to prove a \emph{uniform nesting
  property} for finite-state automata which ``code points'' with
respect to a $\Gamma$-action on some Hausdorff space $M$. Special
cases of this result were originally stated as Lemma~3.8 and
Corollary~3.11 in \cite{MMW1}. We use essentially the same proof as in
that paper to show the general statement (Lemma~\ref{lem:uniformnest}
below).

We start with some general set-up. Let $G$ be any group acting by
homeomorphisms on a Hausdorff space $M$. As in
Section~\ref{sec:automaton_prop} of this paper, for any edge $e$ in a
directed graph, we let $\head(e)$ and $\tail(e)$ respectively denote
the initial and terminal vertices of $e$.

\begin{definition}
  A \emph{finitary point coder} $\mc{Q}$ for the action $G\acts M$ 
  consists of
  \begin{enumerate}
  \item A finite collection  $\mc{W}(\mc{Q})$ of open sets of $M$.
  \item A finite collection $F(\mc{Q})$ of elements of $G$.
  \item A finite directed graph, with each vertex $v$ labeled by an
    open set $W(v) \in \mc{W}(\mc{Q})$, and each edge $e$ labeled by
    an element $\lab(e) \in F(\mc{Q})$, satisfying the following
    conditions:
    \begin{enumerate}
    \item Whenever there is an edge from $z_1$ to $z_2$ labeled by
      $\alpha$, there is an inclusion
      $\alpha\overline{W(z_2)} \subset W(z_1)$.
    \item\label{subitm:point} Whenever $\coding{e} = (e_k)_{k \in \bN}$ is an
      infinite edge path with terminal vertex sequence
      $\tail(e_k) = z_k$ and edge labels $\alpha_k = \lab(e_k)$, the
      sequence of sets
      $\left(\alpha_1 \cdots \alpha_n W(z_n)\right)_{n \in \bN}$ is a
      system of neighborhoods for a point $p \in M$.  Such an edge
      path is called a \emph{strict $\mc{Q}$--coding} of the point
      $p$.
    \end{enumerate}
  \end{enumerate}
\end{definition}

If $\coding{e} = (e_k)_{k \in \bN}$ is a strict $\mc{Q}$-coding of $p$
with labels $\lab(e_k) = \alpha_k$ and terminal vertices
$v_k = \tail(e_k)$, it follows immediately that the intersection
\[
  \bigcap_{n \in \bN} \alpha_1 \cdots \alpha_n W(v_n)
\]
is equal to $\{p\}$. If $M$ is metrizable and every set
$W \in \mc{W}(\mc{Q})$ has compact closure (which is the case for all
of our applications), then the converse also holds: if
$\coding{e} = (e_k)_{k \in \bN}$ is an edge path with label sequence
$(\alpha_k)_{k \in \bN}$, and
$\bigcap_{n \in \bN} \alpha_1 \cdots \alpha_n W(\tail(e_n)) = \{p\}$,
then $\coding{e}$ is a $\mc{Q}$-coding of $p$.

\begin{definition} Let $\mc{Q}$ be a finitary point coder. A
  \emph{generalized $\mc{Q}$--coding} is a pair $(g_0, \coding{e})$,
  where $g_0 \in G$ and $\coding{e}$ is a strict $\mc{Q}$-coding. If
  $\coding{e}$ is a strict coding of $p \in M$, then we say that the
  generalized coding $(g_0, \coding{e})$ codes the point $g_0p \in M$.

  The label sequence and initial/terminal vertex sequences of a
  generalized coding $(g_0, \coding{e})$ are defined to be the same as
  the corresponding sequences for the strict coding $\coding{e}$. If
  $(\alpha_k)_{k \in \bN}$ is the label sequence for a generalized
  coding $(g_0, \coding{e})$, then the \emph{path sequence} associated to the coding is the sequence
  $(g_k)_{k \in \bN \cup \{0\}}$ in $G$ defined by
  \[
    g_k := g_0 \cdot \alpha_1 \cdots \alpha_k.
  \]
  
\end{definition}

We prove the following:

\begin{lemma}[Uniform nesting for finitary point coders]
  \label{lem:uniformnest}
  For any finite subset $F \subset G$ and any finitary point coders
  $\mc{Q}, \mc{Q'}$, there is a number $N = N(\mc{Q}, \mc{Q'}, F)$
  satisfying the following. Suppose that $(g_0, \coding{c})$ is a
  generalized $\mc{Q'}$--coding of $p$ with path sequence
  $(g_k)_{k \in \bN \cup \{0\}}$ and terminal vertex sequence
  $(z_k)_{k \in \bN}$, and $(h_0, \coding{d})$ is a generalized
  $\mc{Q}$--coding of $p$ with path sequence 
  $(h_k)_{k \in \bN \cup \{0\}}$ and terminal vertex sequence
  $(y_k)_{k \in \bN}$. Then, for any indices $m, n \in \bN$ satisfying
  $g_n^{-1}h_m \in F$, we have
  \begin{equation}\label{eq:nestingcond}
    g_{n+N} \overline{W(z_{n+N})} \subset h_{m} W(y_{m}).
  \end{equation}
\end{lemma}
\begin{proof}
  Fix a finite set $F \subset G$ and finitary point coders $\mc{Q}$,
  $\mc{Q'}$. The proof is by contradiction, so we therefore assume we
  have a sequence of natural numbers $N$ tending to $\infty$ so that
  for each $N$ in the sequence, there is a point $p^{N}\in M$, a
  generalized $\mc{Q'}$--coding $(g_0^{(N)}, \coding{c}^{(N)})$ of
  $p^{(N)}$ with path sequence
  $(g_k^{(N)})_{k \in \bN \cup \{0\}}$ and terminal vertex sequence
  $(z_k^{(N)})_{k \in \bN}$, a generalized $\mc{Q}$--coding
  $(h_0^{(N)}, \coding{d}^{(N)})$ of $p^{(N)}$ with path sequence
   $(h_k^{(N)})_{k \in \bN \cup \{0\}}$ and terminal vertex
  sequence $(y_k^{(N)})_{k \in \bN}$, and indices $m_N, n_N \in \bN$
  so that $(g^{(N)}_{n_N})^{-1} h^{(N)}_{m_N} \in F$, but the
  inclusion \eqref{eq:nestingcond} fails. That is, for each $N$, we
  have
  \begin{equation}
    \label{eq:notsubset}
    g^{(N)}_{n_N+N} \overline{W(z_{n_N+N}^{(N)})}
    \not\subset h^{(N)}_{m_N} W(y_{m_N}^{(N)}).
  \end{equation}
  
  We immediately pass to a subsequence so that
  \begin{equation}\tag{f}
    \label{cond:f}
  \mbox{   $ (g_{n_N}^{(N)})^{-1} h_{m_N}^{(N)}$ is constant, equal to $f\in F$. }
  \end{equation}
  We further refine our subsequence so the following three conditions are satisfied.
\begin{align}
  \tag{z} & \mbox{ The sets $W(z_{n_N+N}^{(N)})$ are constant, equal to
            some $W \in \mc{W}(\mc{Q'})$.}\\
  \tag{y} & \mbox{ The sets $W(y_{m_N}^{(N)})$ are constant, equal to
            some $U \in \mc{W}(\mc{Q})$. }\\
  \tag{$\ast$}\label{cond:W} & \mbox{ The sets $(h_{m_N}^{(N)})^{-1}h_{m_N + 1}^{(N)} W(y_{m_N + 1}^{(N)})$ are constant, equal to some $U'$.}
  \end{align}
  The first two are possible because the sets $\mc{W}(\mc{Q'})$ and
  $\mc{W}(\mc{Q})$ are finite; the third also uses the fact that each
  $\mathbf{d}^{(N)}$ is a (strict) $\mc{Q}$--coding, so the elements
  $(h_{m_N}^{(N)})^{-1}h_{m_N + 1}^{(N)} $ lie in the finite set
  $F(\mc{Q})$.  A key property we will use at the end of the proof is
  that
  \begin{equation}\label{eq:UprimeinU}
    \overline{U}'\subset U.
  \end{equation}

  The fact that $\mathbf{c}^{(N)}$ is a $\mc{Q'}$--coding implies that
  for each $k\ge 1$, we have $g_k^{(N)} = g_{k-1}^{(N)}\alpha$ for
  some $\alpha$ chosen from the finite set $F(\mc{Q'})$.  For $N$ in
  our subsequence, we can multiply each side of~\eqref{eq:notsubset}
  on the left by $(g^{(N)}_{n_N})^{-1}$ and apply condition
  \eqref{cond:f} to obtain
  \begin{equation}
    \label{eq:simple}
    \alpha_1^{(N)}\cdots\alpha_N^{(N)} \overline{W} \not\subset f U,
  \end{equation}
  where $\alpha_k^{(N)} := (g_{n_N+k-1}^{(N)})^{-1}g_{n_N + k}^{(N)}$
  is the label of an edge $e_k^{(N)}$ in $\mc{Q'}$.

  For each $N$ we consider the infinite edge path $\gamma^{(N)}$ in
  $\mc{Q'}$ given by
  \[
    \gamma^{(N)} := (e_1^{(N)}, e_2^{(N)}, \ldots)
  \]
  We note that $\gamma^{(N)}$ is a strict $\mc{Q'}$--coding, coding the
  point $(g^{(N)}_{n_N})^{-1}p^N$.

  Passing to a subsequence $\{N(j)\}_{j\in \bN}$ one final time, we
  obtain a sequence of strict $\mc{Q'}$--codings
  $\{\gamma^{(N(j))}\}_{j\in \bN}$ so that for all $l\ge j$, the
  initial subsegment of length $j$ of the $\mc{Q'}$-coding
  $\gamma^{N(l)}$ is independent of $l$. In particular this
  subsequence of codings converges to a strict $\mc{Q'}$--coding
  \[ \gamma^\infty = (e_1^\infty, e_2^\infty, \ldots ) \] with edge
  labels $\alpha_k^\infty := \lab(e_k^\infty)$. By property
  \eqref{subitm:point} of $\mc{Q'}$, this coding determines a unique
  point $p^\infty \in M$.

For our subsequence $N(j)$, the non-containment in \eqref{eq:simple} takes the form
\begin{equation*}
 \alpha_{1}^{(N(j))}\cdots\alpha_{N(j)}^{(N(j))} \overline{W} \not\subset f U.
\end{equation*}
We may assume that $N(j) > j$ for all $j$, so we can rewrite this as
\begin{equation*}
  \left(\alpha_1^\infty\cdots\alpha_j^\infty\right)\left(\alpha_{j+1}^{(N(j))}\cdots \alpha_{N(j)}^{(N(j))}\right) \overline{W} \not \subset f U.
\end{equation*}
 By the nesting property of codings we have
 \begin{equation*}
   \left(\alpha_1^\infty\cdots\alpha_j^\infty\right)\left(\alpha_{j+1}^{(N(j))}\cdots \alpha_{N(j)}^{(N(j))}\right) \overline{W} \subset \left(\alpha_1^\infty\cdots\alpha_j^\infty\right) W^\infty_j,
   \end{equation*}
   so we must therefore have
\begin{equation}\label{eq:pinftynbhds}
     \left(\alpha_1^\infty\cdots\alpha_j^\infty\right) W^\infty_j \not\subset fU.
   \end{equation}
   Since $\gamma^\infty$ is a coding for $p^\infty$, the sets on the
   left hand side of~\eqref{eq:pinftynbhds} give a nested basis of
   neighborhoods of $p^\infty$ and so we conclude
   \begin{equation}
     \label{eq:notinrhs}
     p^\infty\not\in fU.
   \end{equation}

On the other hand, for each $N$,  $\gamma^{(N)}$ is a coding for $(g^{(N)}_{n_N})^{-1}p^N$.  Thus for each $j$ we have
\begin{align*}
  (g^{(N(j))}_{n_{N(j)}})^{-1} p^{N(j)}
  \in & \alpha_1^{(N(j))}\cdots \alpha_j^{(N(j))}W^{(N(j))}_j \\
  & = \alpha_1^\infty\cdots \alpha_j^\infty W^\infty_j.
\end{align*}
As before this last sequence of sets is a nested neighborhood basis for $p^\infty$ and
thus
\begin{equation*}
  \lim_{j\to\infty} (g^{(N(j))}_{n_{N(j)}})^{-1} p^{N(j)} = p^{\infty}.
\end{equation*}
We also know that, for any $N$,
\begin{equation}\label{eq:gnpn}
  (g^{(N)}_{n_N})^{-1} p^{N}\in (g^{(N)}_{n_N})^{-1} h^{(N)}_{m_N+1}\overline{W(y_{m_N + 1}^{(N)})},
\end{equation}
since $\mathbf{d}^{(N)}$ also codes $p^N$.
By our assumptions
\eqref{cond:f} and \eqref{cond:W} on our chosen subsequence, the
right-hand side of~\eqref{eq:gnpn} is always equal to a constant
$fU'$.  
   But we have just seen that a subsequence of the left-hand side
converges to $p^\infty$, so we must have $p^{\infty}\in f\overline{U'}$.  Because of~\eqref{eq:UprimeinU} this implies
\begin{equation*}
  p^{\infty}\in f U,
\end{equation*}
contradicting~\eqref{eq:notinrhs}.
\end{proof}

\bibliography{../refs.bib}
\bibliographystyle{shortalpha}
\end{document}